%% file: 3manifolds_everywhere.tex
\documentclass{amsart}
\usepackage{pinlabel}
\usepackage{amsmath, amsthm, amssymb, amscd, mathrsfs, eucal, epsfig,color}
\usepackage{hyperref}

\input xy
\xyoption{all}

\begin{document}

\newtheorem{theorem}{Theorem}[subsection]
\newtheorem{lemma}[theorem]{Lemma}
\newtheorem{corollary}[theorem]{Corollary}
\newtheorem{conjecture}[theorem]{Conjecture}
\newtheorem{proposition}[theorem]{Proposition}
\newtheorem{question}[theorem]{Question}
\newtheorem*{answer}{Answer}
\newtheorem{problem}[theorem]{Problem}
\newtheorem*{main_theorem}{3-Manifolds Everywhere Theorem~\ref{theorem:random_acylindrical_subgroup}}
\newtheorem*{boundary_theorem}{Boundary Theorem~\ref{theorem:boundary_theorem}}
\newtheorem*{commensurability_theorem}{Commensurability Theorem~\ref{theorem:commensurability_theorem}}
\newtheorem*{thin_spine_theorem}{Thin Spine Theorem~\ref{theorem:thin_spine_theorem}}
\newtheorem*{claim}{Claim}
\newtheorem*{criterion}{Criterion}
\theoremstyle{definition}
\newtheorem{definition}[theorem]{Definition}
\newtheorem{construction}[theorem]{Construction}
\newtheorem{notation}[theorem]{Notation}
\newtheorem{convention}[theorem]{Convention}
\newtheorem*{warning}{Warning}

\theoremstyle{remark}
\newtheorem{remark}[theorem]{Remark}
\newtheorem{example}[theorem]{Example}
\newtheorem{scholium}[theorem]{Scholium}
\newtheorem*{case}{Case}

\def\id{\text{id}}
\def\Id{\text{Id}}
\def\1{{\bf{1}}}
\def\p{{\mathfrak{p}}}
\def\H{\mathbb H}
\def\Z{\mathbb Z}
\def\R{\mathbb R}
\def\C{\mathbb C}
\def\F{\mathbb F}
\def\P{\mathbb P}
\def\Q{\mathbb Q}
\def\E{{\mathcal E}}
\def\D{{\mathcal D}}
\def\A{{\mathcal A}}
\def\I{{\mathcal I}}
\def\density{\textnormal{density}}

\def\tra{\textnormal{tr}}
\def\length{\textnormal{length}}

\def\cube{\textnormal{cube}}

\newcommand{\marginal}[1]{\marginpar{\tiny #1}}

\title{3-Manifolds Everywhere}

\author{Danny Calegari}
\address{Department of Mathematics \\ University of Chicago \\
Chicago, Illinois, 60637}
\email{dannyc@math.uchicago.edu}
\author{Henry Wilton}
\address{DPMMS\\Centre for Mathematical Sciences\\Wilberforce Road\\Cambridge\\CB3 0WB\\UK}
\email{h.wilton@maths.cam.ac.uk}

\begin{abstract}
A random group contains many subgroups which are isomorphic to the
fundamental group of a compact hyperbolic 3-manifold with totally geodesic boundary. 
These subgroups can be taken to be quasi-isometrically embedded. This is true 
both in the few relators model, and the density model of random groups (at any
density less than a half).
\end{abstract}

\iffalse
\fi

\maketitle

\setcounter{tocdepth}{1}
\tableofcontents

\section{Introduction}

Geometric group theory was born in low-dimensional topology, in the collective visions of
Klein, Poincar\'e and Dehn. Stallings used key ideas from 3-manifold
topology (Dehn's lemma, the sphere theorem) to prove theorems about free groups, and as a model for
how to think about groups geometrically in general. The pillars of modern geometric group theory ---
(relatively) hyperbolic groups and hyperbolic Dehn filling, NPC cube complexes and their relations to
LERF, the theory of JSJ splittings of groups and the structure of limit groups --- all have their
origins in the geometric and topological theory of 2- and 3-manifolds.

Despite these substantial and deep connections, the role of 3-manifolds in the larger
world of group theory has been mainly to serve as a source of examples --- of specific groups, and
of rich and important phenomena and structure. Surfaces (especially Riemann surfaces)
arise naturally throughout all of mathematics (and throughout science more generally), 
and are as ubiquitous as the complex numbers. But the conventional view is surely that
3-manifolds {\it per se} do not spontaneously arise in other areas of geometry (or mathematics
more broadly) amongst the {\em generic} objects of study. We challenge this conventional view:
3-manifolds are everywhere.

\subsection{Random groups}

The ``generic'' objects in the world of finitely presented groups are the {\em random} groups, in the
sense of Gromov. There are two models of what one means by a random group, and we shall
briefly discuss them both.

First, fix $k\ge 2$ and fix a free generating set $x_1,x_2,\cdots,x_k$ for $F_k$, a free group of
rank $k$. A $k$-generator group $G$ can be given by a presentation
$$G:=\langle x_1,x_2,\cdots,x_k\; | \; r_1,r_2,\cdots,r_\ell\rangle$$
where the $r_i$ are cyclically reduced cyclic words in the $x_i$ and their inverses.

In the {\em few relators model} of a random group, one fixes $\ell \ge 1$, and then for a
given integer $n$, selects the $r_i$ independently and randomly (with the uniform distribution) from
the set of all reduced cyclic words of length $n$.

In the {\em density model} of a random group, one fixes $0 < D < 1$, and then for a given integer
$n$, define $\ell = \lfloor (2k-1)^{Dn} \rfloor$ and select the $r_i$ independently and randomly
(with the uniform distribution) from the set of all reduced cyclic words of length $n$.

Thus, the difference between the two models is how the number of relations ($\ell$) depends on
their length ($n$). In the few relators model, the absolute number of relations is fixed, whereas
in the density model, the (logarithmic) density of the relations among all words of the given length
is fixed.

For fixed $k,\ell,n$ in the few relators model, or fixed $k,D,n$ in the density model, we obtain
in this way a probability distribution on finitely presented groups (actually, on finite presentations).
For some property of groups of interest, the property will hold for a random group with some
probability depending on $n$. We say that the property holds for a random group 
{\em with overwhelming probability} if the probability goes to 1 as $n$ goes to infinity.

As remarked above, a ``random group'' really means a ``random presentation''. Associated to a
finite presentation of a group $G$ as above, one can build a 2-complex $K$ with one 0 cell, with
one 1 cell for each generator $x_i$, and with one 2 cell for each relation $r_j$, so that $\pi_1(K)=G$.
We are very interested in the geometry and combinatorics of $K$ (and its universal cover) in what follows.

\subsection{Properties of random groups}

All few-relator random groups are alike (with overwhelming probability).
There is a {\em phase transition} in the behavior of density random groups, discovered by
\cite{Gromov} \S~9: for $D>1/2$, a random group is either trivial or isomorphic to $\Z/2\Z$,
whereas at any fixed density $0<D<1/2$, a random group is infinite, hyperbolic, and 1-ended, and the
presentation determining the group is {\em aspherical} --- i.e.\/ the 2-complex $K$ defined
from the presentation has contractible universal cover. Furthermore,
\cite{Dahmani_Guirardel_Przytycki} showed that a random group
with density less than a half does not split, and has boundary homeomorphic to the {\em Menger sponge}.

\begin{figure}[htpb]
\labellist
\small\hair 2pt
%\pinlabel $\text{Type A}$ at -200 0
\endlabellist
\centering
\includegraphics[scale=1]{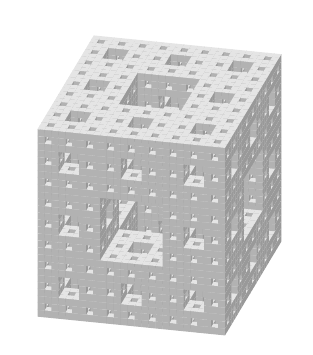}
\caption{The boundary of a random group is a Menger sponge}\label{sponge}
\end{figure}

The Menger sponge is obtained from the unit cube by subdividing it into 27 smaller cubes each
with one third the side length, then removing the central cube and the six cubes centered on each face;
and then inductively performing the same procedure with each of the remaining 20 smaller cubes; see
Figure~\ref{sponge}.

The Menger sponge has topological dimension 1, and is the {\em universal} compact space
with this property, in the sense that any compact Hausdorff space of topological dimension 1
embeds in it.

It is important to understand what kinds of abstract groups $H$ arise as subgroups of a random
group $G$. However, not all subgroups of a hyperbolic group are of equal importance: the most useful
subgroups, and those that tell us the most about the geometry of $G$, are the subgroups $H$ with the
following properties:
\begin{enumerate}
\item{the group $H$ itself is something whose intrinsic topological and 
geometric properties we understand very well; and}
\item{the intrinsic geometry of $H$ can be uniformly compared with the extrinsic geometry of its 
embedding in $G$.}
\end{enumerate}
In other words, we are interested in well-understood groups $H$ which are quasi-isometrically embedded
in $G$.

A finitely generated quasi-isometrically embedded subgroup $H$ of a hyperbolic group $G$ is itself hyperbolic,
and therefore finitely presented. The inclusion of $H$ into $G$ induces an embedding of Gromov boundaries
$\partial_\infty H \to \partial_\infty G$. Thus if $G$ is a random group, the boundary of $H$ has
dimension at most 1, and by work of \cite{Kapovich_Kleiner}, there is a
hierarchical description of all possible $\partial_\infty H$.

First of all, if $\partial_\infty H$ is disconnected, then one knows by
\cite{Stallings} that $H$ splits over a finite group. Second of all, if $\partial_\infty H$
is connected and contains a local cut point, then one knows by \cite{Bowditch} that 
either $H$ is virtually a surface group, or else $H$ splits over a cyclic group. Thus, apart from the
Menger sponge itself, understanding hyperbolic groups with boundary of dimension at most 1
reduces (in some sense) to the case that $\partial_\infty H$ is a Cantor set,
a circle, or a {\em Sierpinski carpet} --- i.e.\/ one of the ``faces'' of the Menger cube. The
Sierpinski carpet is universal for 1-dimensional compact Hausdorff {\em planar} sets of topological
dimension 1. Thus one is naturally led to the following fundamental question, which as far as we
know was first asked explicitly by Fran\c cois Dahmani:

\begin{question}[Dahmani]
Which of the three spaces --- Cantor set, circle, Sierpinski carpet --- arise as the boundary of
a quasiconvex subgroup of a random group?
\end{question}

Combining the main theorem of this paper with the results of \cite{Calegari_Walker} gives a complete
answer to this question: 

\begin{answer}
{\em  All three} spaces arise in a random group as the boundary of a quasiconvex subgroup
(with overwhelming probability).
\end{answer}
 
This is true in the few relators model with any positive number of relators, or the density model at any density
less than a half. The situation is summarized as follows:

\begin{enumerate}
\item{$\partial_\infty H$ is a Cantor set if and only if $H$ is virtually free (of rank at least 2); 
we can thus take
$H = \pi_1(\text{graph})$. The existence of free quasiconvex subgroups $H$ in arbitrary (nonelementary)
hyperbolic groups $G$ is due to Klein, by the ping-pong argument.}
\item{$\partial_\infty H$ is a circle if and only if $H$ is virtually a surface group (of genus
at least 2); we can thus take $H=\pi_1(\text{surface})$. The existence of surface subgroups $H$ in
random groups (with overwhelming probability) is proved by \cite{Calegari_Walker}, Theorem~6.4.1.}
\item{$\partial_\infty H$ is a Sierpinski carpet if $H$ is virtually the fundamental group of
a compact hyperbolic 3-manifold with (nonempty) totally geodesic boundary;
and if the Cannon conjecture is true, this is if and only if --- see \cite{Kapovich_Kleiner}. 
The existence of such 3-manifold subgroups $H$ in random groups (with overwhelming probability) 
is Theorem~\ref{theorem:random_acylindrical_subgroup} in this paper.}
\end{enumerate}

Explicitly, we prove:

\begin{main_theorem}
Fix $k\ge 2$. A random $k$-generator group --- either in the few relators model
with $\ell \ge 1$ relators, or the density model with density $0<D<1/2$ --- and relators
of length $n$ contains many quasi-isometrically embedded subgroups isomorphic to the fundamental group 
of a hyperbolic 3-manifold with totally geodesic boundary, with probability $1-O(e^{-n^C})$ for some $C>0$.
\end{main_theorem}

Our theorem applies in particular in the few relators model where $\ell = 1$. In fact, if one fixes
$D<1/2$, and starts with a random 1-relator group $G = \langle F_k \; | \; r \rangle$, one can
construct a quasiconvex 3-manifold subgroup $H$ (of the sort which is guaranteed by the theorem) which
stays injective and quasiconvex (with overwhelming probability) as a further $(2k-1)^{Dn}$ relators are added.

\subsection{Commensurability}

Two groups are said to be {\em commensurable} if they have isomorphic subgroups of finite index.
Commensurability is an equivalence relation, and it is natural to wonder what commensurability
classes of 3-manifold groups arise in a random group.

We are able to put very strong constraints on the commensurability classes of the
3-manifold groups we construct. It is probably too much to hope to be able to
construct a subgroup of a {\em fixed} commensurability class. But we can arrange
for our 3-manifold groups to be commensurable with some element of a family of finitely generated
groups given by presentations which differ only by varying the order of torsion of a specific element.
Hence our 3-manifold subgroups are all commensurable
with Kleinian groups of bounded convex covolume (i.e.\/ the convex hulls of the
quotients have uniformly bounded volume). Explicitly:

\begin{commensurability_theorem}
A random group at any density $<1/2$ or in the few relators model contains 
(with overwhelming probability) a subgroup commensurable with the Coxeter group $\Gamma(m)$
for some $m\ge 7$, where $\Gamma(m)$ is the Coxeter group with Coxeter diagram
\begin{center}
\begin{picture}(70,10)(0,0)
\put(5,5){\circle*{3}}
\put(35,5){\circle*{3}}
\put(65,5){\circle*{3}}
\put(95,5){\circle*{3}}
\put(5,5){\line(1,0){90}}
\put(16,9){$m$}
\end{picture}
\end{center}
\end{commensurability_theorem}

The Coxeter group $\Gamma(m)$ is commensurable with the group generated by reflections in the
sides of a regular {\em super ideal} tetrahedron --- one with vertices ``outside'' the sphere
at infinity --- and with dihedral angles $\pi/m$. 

\subsection{Plan of the paper}

We now describe the outline of the paper. 

As discussed above, our random groups $G$ come together with the data of a finite presentation 
$$G:=\langle x_1,\cdots,x_k \;|\; r_1,\cdots,r_\ell\rangle$$
From such a presentation we can build in a canonical way a 2-complex $K$ with 1-skeleton $X$, where
$X$ is a wedge of $k$ circles, and $K$ is obtained by attaching $\ell$ disks along loops in $X$
corresponding to the relators; so $\pi_1(K) = G$.

Our 3-manifold subgroups arise as the fundamental groups of 2-complexes $\overline{M}(Z)$ that
come with immersions $\overline{M}(Z) \to K$ taking (open) cells of $\overline{M}(Z)$ homeomorphically
to cells of $K$. Thus, for every 2-cell of $\overline{M}(Z)$, the attaching map of its boundary
to the 1-skeleton $Z$ factors through a map onto one of the relators of the given presentation of $G$.

One way to obtain such a complex $\overline{M}(Z)$ is to build a 1-complex $Z$ as a quotient
of a collection $L$ of circles together with an immersion $Z \to X$ where $X$ is the 1-skeleton
of $K$, and the map $L \to X$ takes each component to the image of a relator. We call data of this
kind a {\em spine}. In \S~\ref{section:spines} we describe the topology of spines and give
sufficient combinatorial conditions on a spine to ensure that $\overline{M}(Z)$ is 
homotopic to a 3-manifold. 

Since  $X$ is a rose whose edges are endowed with a choice of orientation and labelling by the generators $x_i$, we will usually encode a map of graphs $\Gamma\to X$ by labelling (oriented) edges by $x_i^{\pm1}$, in the spirit of \cite{Stallings_graphs}.  As usual, if an oriented edge $e$ of $\Gamma$ is labelled $x_i^{\pm1}$ then the oriented edge $\bar{e}$ with the reverse orientation is labelled $x_i^{\mp1}$.  Note that there is a simple condition to ensure that such a map $\Gamma\to X$ is an immersion: one simply requires that no two oriented edges of $\Gamma$ incident at the same vertex have the same label.  We call such a graph $\Gamma$ \emph{folded} (also in the spirit of \cite{Stallings_graphs}).  

In \S~\ref{section:thin_spine_section} we prove the Thin Spine Theorem,
which says that we can build such
a spine $L \to Z$, satisfying the desired combinatorial conditions,
and such that every edge of the 1-skeleton $Z$ is {\em long}. Here we measure
the length of edges of $Z$ by pulling back length from $X$ under the immersion $Z\to X$, 
where  each edge of $X$ is normalized to have length $1$. 
In fact, if we let $G_1$ denote the 1-relator group 
$$G_1:=\langle x_1,\cdots,x_k\;|\; r_1\rangle$$
with associated 2-complex $K_1$ which comes with a tautological inclusion
$K_1 \to K$, then our thin spines have the property that the
immersion $\overline{M}(Z) \to K$ factors through $\overline{M}(Z) \to K_1$. 

For technical reasons, rather than working with a random relator $r_1$, 
we work instead with a relator which is merely sufficiently {\em pseudorandom} 
(a condition concerning equidistribution of subwords with controlled error on certain scales),
and the theorem we prove is deterministic. Of course, the definition of pseudorandom is such
that a random word will be pseudorandom with very high probability.

Explicitly, we prove:

\begin{thin_spine_theorem}
For any $\lambda>0$ there is $T\gg \lambda$ and $\epsilon\ll 1/T$ so that, 
if $r$ is $(T,\epsilon)$-pseudorandom and $K$ is the 2-complex associated 
to the presentation $G:=\langle F_k\;|\; r\rangle$, then there is a spine 
$f:L \to Z$ over $K$ for which $L$ is a union of 648 circles (or 5,832 circles if $k=2$), and every edge 
of $Z$ has length at least $\lambda$.
\end{thin_spine_theorem}

This is by far the longest section in the paper, and it involves
a complicated combinatorial argument with many interdependent steps. It should be remarked that
one of the key ideas we exploit in this section is the method of {\em random matching with
correction}: randomness (actually, pseudorandomness) is used to show that the desired combinatorial
construction can be performed {\em with very small error}. In the
process, we build a {\em reservoir} of small independent pieces which may be adjusted by various
local moves in such a way as to ``correct'' the errors that arose at the random matching step.
Similar ideas were also used by \cite{Kahn_Markovic} in their proof of the
Ehrenpreis Conjecture, by \cite{Calegari_Walker} in their construction of
surface subgroups in random groups, and by \cite{Keevash} in his construction of 
General Steiner Systems and Designs. Evidently this method is extremely powerful, and its
full potential is far from being exhausted.

The Thin Spine Theorem can be summarized by saying that as a graph, 
$Z$ has bounded valence, but very long edges. This means that
the image of $\pi_1(Z)$ in $\pi_1(X)$ induced by the inclusion $Z \to X$ is very ``sparse'', in 
the sense that the ball of radius $n$ in $\pi_1(X)$ contains $O(3^{n/\lambda})$ elements of
$\pi_1(Z)$, where we can take $\lambda$ as big as we like. This has the following consequence:
when we obtain $G_1$ as a quotient of $\pi_1(X)$ by adding $r_1$ as a relator, we should not kill any
``accidental'' elements of $\pi_1(Z)$, so that the image of $\pi_1(Z)$ will be isomorphic to
$\pi_1(\overline{M}(Z))$, a 3-manifold group. 

This idea is fleshed out in 
\S~\ref{bead_decomposition_section}, and shows that random 1-relator groups contain 3-manifold
subgroups, although at this stage we have not yet shown that the 3-manifold is of the desired
form. The argument in this section depends on a so-called {\em bead decomposition}, which
is very closely analogous to the bead decomposition used to construct surface groups by
\cite{Calegari_Walker}, and the proof is very similar.

In \S~\ref{acylindricity_section} we show that the 3-manifold homotopic to the 2-complex
$\overline{M}(Z)$ is acylindrical; equivalently, that it is homeomorphic to a compact
hyperbolic 3-manifold with totally geodesic boundary.
This is a step with no precise analog in \cite{Calegari_Walker},
but the argument is very similar to the argument showing that $\overline{M}(Z)$ is injective in the
1-relator group $G'$. There are two kinds of annuli to rule out: those that use 2-cells of
$\overline{M}(Z)$, and those that don't. The annuli without 2-cells are ruled out by the
combinatorics of the construction. Those that use 2-cells are ruled out by a small cancellation
argument which uses the thinness of the spine. So at the end of this section, we have shown
that random 1-relator groups contain subgroups isomorphic to the fundamental groups of
compact hyperbolic 3-manifolds with totally geodesic boundary.

Finally, in \S~\ref{endgame_section}, we show that the subgroup $\pi_1(\overline{M}(Z))$ stays
injective as the remaining $\ell-1$ random relators are added. The argument here stays extremely
close to the analogous argument in \cite{Calegari_Walker}, and depends (as \cite{Calegari_Walker} did)
on a kind of small cancellation theory for random groups developed by \cite{Ollivier}. 
This concludes the proof of the main theorem.

A further section \S~\ref{section:commensurability} 
proves the Commensurability Theorem. 
The proof is straightforward given the technology  
developed in the earlier sections.

\section{Spines}\label{section:spines}

\subsection{Trivalent fatgraphs and spines}

Before introducing spines, we first motivate them by describing the analogous, but simpler,
theory of fatgraphs.

A {\em fatgraph} $Y$ is a (simplicial) graph together with a cyclic ordering of the edges incident to
each vertex. This cyclic ordering can be used to canonically ``fatten'' the graph $Y$ so that it
embeds in an oriented surface $S(Y)$ with boundary, in such a way that $S(Y)$ deformation retracts 
down to $Y$. Under this deformation retraction, the boundary $\partial S(Y)$ maps to $Y$ in such a
way that the preimage of each edge $e$ of $Y$ consists of two intervals $e^\pm$ 
in $\partial S(Y)$, each mapping homeomorphically to $e$, with opposite orientations.

Abstractly, the data of a fatgraph can be given by an ordinary graph $Y$, a 1-manifold $L$, and
a locally injective simplicial map $f:L \to Y$ of (geometric) ``degree 2'';
i.e.\/ such that each edge of $Y$ is in the image of
two intervals in $L$. The surface $S(Y)$ arises as the mapping cone of $f$.
If one orients $L$ and insists that the preimages of each edge have opposite
orientations, the result is a fatgraph and an oriented surface as above. If one does not 
insist on the orientation condition, the mapping cone need not be orientable. Attaching a disk
along its boundary to each component of $L$ produces a {\em closed} surface, which we denote
$\overline{S}(Y)$.

We would like to discuss a more complicated object called a {\em spine}, 
for which the analog of $\overline{S}(Y)$ is a 2-complex
homotopy equivalent to a compact 3-manifold with boundary. The 2-complex will arise by
gluing 2-dimensional disks onto the components of a 1-manifold $L$, and then attaching these disks
to the mapping cone of an immersion $f:L \to Z$ where $Z$ is a 4-valent graph, and the map
$f$ is subject to certain local combinatorial constraints.

The first combinatorial constraint is that the map $f:L \to Z$ should be ``degree 3''; that is, the preimage
of each edge of $L$ should consist of three disjoint intervals in $L$, each mapped homeomorphically by $f$.

Since $Z$ is 4-valent, at each vertex $v$ of $Z$ we have 12 intervals in $L$ that map to the incident edges;
these 12 intervals should be obtained by subdividing 6 disjoint intervals in $L$, where the dividing point
maps to $v$. Since $L \to Z$ is an immersion, near each dividing point the given interval in $L$ runs locally
from one edge incident to $v$ to a different one. There are three local models (up to symmetry) of how
six edges of $L$ can locally run over a 4-valent vertex $v$ of $Z$ so that they run over each incident edge
(in $Z$) to $v$ three times (this notion of ``local model'' is frequently called a
{\em Whitehead graph} in the literature). 
These three local models are illustrated in Figure~\ref{local_model}.

\begin{figure}[htpb]
\labellist
\small\hair 2pt
\pinlabel $\text{Type A}$ at -200 0
\endlabellist
\centering
\includegraphics[scale=0.7]{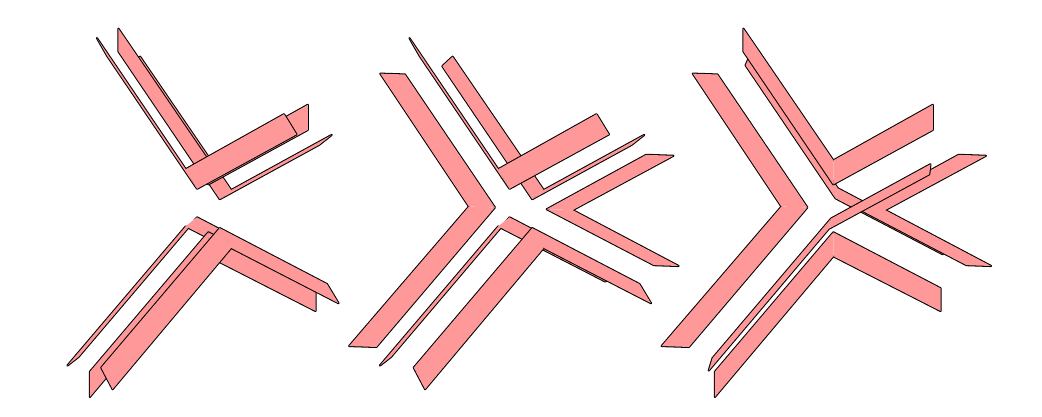}
\caption{Three local models}\label{local_model}
\end{figure}

The third local model is distinguished by the property that for each pair of edges of $Z$ adjacent to
$v$, there is exactly one interval in $L$ running over $v$ from one edge to the other. We say that
such a local model is {\em good}. The second combinatorial constraint is that the local model at every
vertex of $Z$ is good.

If $f:L \to Z$ is good, and $M(Z)$ is the mapping cone of $f$, then the 2-complex $M(Z)$ can be
canonically thickened to a 3-manifold, since a neighborhood of the mapping cone near a
vertex $v$ embeds in $\R^3$ in such a way that the tetrahedral symmetry of the combinatorics is
realized by symmetries of the embedding. Similarly, along each edge the dihedral symmetry is
realized by the symmetry of the embedding. The restriction of this thickening to each
component of $L$ is the total space of an $I$-bundle over the circle; we say that $f:L \to Z$
is {\em co-oriented} if each of these $I$-bundles is trivial. The third combinatorial constraint
is that $f:L \to Z$ is co-oriented.

\begin{definition}\label{definition:spine}
A {\em spine} is the data of a compact 1-manifold $L$ and a 4-valent graph $Z$, together with a 
co-oriented degree 3 immersion $f:L \to Z$ whose local model at every vertex of $Z$ is good. 
If $f:L \to Z$ is a spine, we denote the mapping cone by $M(Z)$, and by $\overline{M}(Z)$ 
the 2-complex obtained by capping each component of $L$ in $M(Z)$ off by a disk.
\end{definition}

\begin{lemma}[Spine thickens]\label{lemma:spine_thickens}
Let $f:L \to Z$ be a spine. Then $\overline{M}(Z)$ is canonically
homotopy equivalent to a compact 3-manifold with boundary.
\end{lemma}
\begin{proof}
We have already seen that $M(Z)$ has a canonical thickening to a compact 3-manifold in such a
way that the restriction of this thickening to each component of $L$ is an $I$-bundle. The total space
of this $I$-bundle is an annulus embedded in the boundary of $M(Z)$, and we may therefore attach
a 2-handle with core the corresponding component of $L$ providing this $I$-bundle is trivial. But
that is exactly the condition that $f:L \to Z$ should be co-oriented.
\end{proof}

By abuse of notation, we call $\overline{M}(Z)$ the {\em thickening} of $Z$.

\subsection{Tautological immersions}

Now let's fix $k\ge 2$ and a free group $F_k$ on $k$ fixed generators. Let $X$ be a rose for $F_k$;
i.e.\/ a wedge of $k$ (oriented) circles, with a given labeling by the generators of $F_k$.
Let $G$ be a random group (in whatever model) at length $n$. Each relator $r_i$ is a cyclically reduced
word in $F_k$, and is realized geometrically by an immersion of an oriented circle $\iota_i:S^1_i \to X$.
Attaching a disk along each such circle gives rise to the 2-complex $K$ described in the introduction
with $\pi_1(K) = G$.

\begin{definition}
A spine {\em over $K$} is a spine $f:L \to Z$ together with an immersion $g:Z \to X$ such that for
each component $L_i$ of $L$, there is some relator $r_j$ and a simplicial homeomorphism 
$g_i:L_i \to S^1_j$ for which $\iota_j g_i = g f$.
\end{definition}

The existence of the simplicial homeomorphisms $g_i$ lets us label the components $L_i$ by the
corresponding relators in such a way that the map $f:L \to Z$ has the property that the preimages of
each edge of $Z$ get the same labels, at least if we choose orientations correctly, and use the convention
that changing the orientation of an edge replaces the label by its inverse. So we can equivalently
think of the labels as living on the oriented edges of $Z$. Notice that the maps $g_i$, if they exist
at all, are {\em uniquely determined} by $g,f,\iota_j$ (at least if the presentation is not redundant,
so that no relator is equal to a conjugate of another relator or its inverse).

Evidently, if $f:L \to Z$, $g:Z \to X$ is a spine over $K$, the immersion $g:Z \to X$ extends
to an immersion of the thickening $\overline{g}:\overline{M}(Z) \to K$, that we call the
{\em tautological immersion}.

Our strategy is to construct a spine over $K$ for which $\overline{M}(Z)$ is
homotopy equivalent to a compact hyperbolic 3-manifold with geodesic boundary, 
and for which the tautological immersion induces a quasi-isometric embedding on $\pi_1$.

\section{The Thin Spine Theorem}\label{section:thin_spine_section}

The purpose of this section is to prove the Thin Spine Theorem, the analog in our context of
the Thin Fatgraph Theorem from \cite{Calegari_Walker}. 

In words, this theorem says that if $r$ is a sufficiently long random cyclically reduced
word in $F_k$ giving rise to a random 1-relator group $G:=\langle F_k\;|\; r\rangle$ with
associated 2-complex $K$, then with overwhelming probability, there is a good spine $f:L \to Z$
over $K$ for which every edge of $Z$ is as long as we like; colloquially, the spine is {\em thin}.

For technical reasons, we prove this theorem merely for sufficiently ``pseudorandom''
words, to be defined presently.

\subsection{Pseudorandomness}\label{subsection:pseudorandomness}

Instead of working directly with random chains, we use a deterministic variant
called {\em pseudorandomness}.

\begin{definition}\label{definition:pseudorandom}
Let $\Gamma$ be a cyclically reduced cyclic word in a free group $F_k$ with $k\ge 2$ generators.
We say $\Gamma$ is {\em $(T,\epsilon)$-pseudorandom} if the following is true: if we pick any
cyclic conjugate of $\Gamma$, and write it as a reduced product of reduced words $\{w_1,\ldots,w_n\}$ of length $T$ (and at most one word $v$ of length $<T$) 
$$\Gamma: = w_1w_2w_3\cdots w_nv$$
(so $n=\lfloor |r|/T\rfloor$) then for every reduced word $\sigma$ of length $T$ in $F_k$, there is an estimate
$$1-\epsilon \le \frac {\# \lbrace i \text{ such that } w_i = \sigma \rbrace} 
n \cdot (2k)(2k-1)^{T-1}\le 1+\epsilon$$
Here the factor $(2k)(2k-1)^{T-1}$ is simply the number of reduced words in $F_k$ of length $T$.
Similarly, we say that a collection of $n$ reduced words $\{w_i\}$ each of length $T$ is 
{\em $\epsilon$-pseudorandom} if for every reduced word $\sigma$ of length $T$ in $F_k$ the
estimate above holds.
\end{definition}

For any $T,\epsilon$, a random reduced word of length $N$ will be $(T,\epsilon)$-pseudorandom
with probability $1-O(e^{-N^c})$ for a suitable constant $c(T,\epsilon)$. This follows
immediately from the standard Chernoff inequality for the stationary Markov process that produces
a random reduced word in a free group  (cf.\ \cite[Lemma 3.2.2]{Calegari_Walker}).

With this definition in place, the statement of the Thin Spine Theorem is:

\begin{theorem}[Thin Spine Theorem]\label{theorem:thin_spine_theorem}
For any $\lambda>0$ there is $T\gg \lambda$ and $\epsilon\ll 1/T$ so that, 
if $r$ is $(T,\epsilon)$-pseudorandom and $K$ is the 2-complex associated 
to the presentation $G:=\langle F_k\;|\; r\rangle$, then there is a spine 
$f:L \to Z$ over $K$ for which $L$ is a union of 648 circles (or 5,832 circles if $k=2$), and every edge 
of $Z$ has length at least $\lambda$.
\end{theorem}

The strange appearance of the number 648 (or 5,832 for $k=2$) in the statement of this theorem reflects the method of proof. First of all, observe that if $f:L \to Z$ is any spine, then since $f$ has degree 3, the total length of $L$ is divisible by 3. If this spine is over $K$, then each component of $L$ has length $|r|$, and if we make no assumptions about the value of $|r|$ mod 3, then it will be  necessary in general for the number of components of $L$ to be divisible by 3. 

Our argument is to gradually glue up more and more of $L$, constructing $Z$ as we go. At an intermediate
stage, the remainder to be glued up consists of a collection of disjoint segments from $L$, and the
power of our method is precisely that this lets us reduce the gluing problem to a collection of
{\em independent} subproblems of {\em uniformly bounded} size. But each of these subproblems must
involve a subset of $L$ of total length divisible by 3 or 6, and therefore it is necessary to ``clear
denominators'' (by taking 2 or 3 disjoint copies of the result of the partial construction) several times
to complete the construction (in the case of rank 2 one extra move might require a further factor
of 9).

Finally, at the last step of the construction, we take 2 copies of $L$ and perform 
a final adjustment to satisfy the co-orientation condition.

The remainder of this section is devoted to the proof of Theorem~\ref{theorem:thin_spine_theorem}.

\subsection{Graphs and types}\label{subsection:graphs_and_types}

Let $L$ be a labeled graph consisting of 648 disjoint cycles (or 5,832 disjoint cycles if $k=2$), 
each labeled by $r$.  We will build the spine $Z$ and the map  $f:L \to Z$ in stages.   We think of $Z$ as a quotient space of $L$, obtained by 
identifying segments in $L$ with the same labels. So the construction of $Z$ proceeds 
by inductively identifying more and more segments of $L$, so that at each stage some 
portion of $L$ has been ``glued up'' to form part of the graph $Z$, and some remains 
still unglued. 

We introduce the following notation and terminology.   Let $\Lambda_0$ be a single cycle 
labeled $r$.  At the $i$th stage of our construction, we deal with a partially glued 
graph $\Lambda_i$, constructed from a certain number of copies of $\Lambda_{i-1}$ via 
certain ungluing and gluing moves.  At each stage, $\Lambda_i$ is equipped with a labelling, defining a map $\Lambda_i\to X$.  We will always be careful to ensure that $\Lambda_i$ is folded, i.e.\ that the map $\Lambda_i\to X$ is an immersion.  The glued subgraph of $\Lambda_i$ is denoted by 
$\Gamma_i$, and the unglued subgraph by $\Upsilon_i$.  Shortly, the unglued subgraph 
$\Upsilon_i$ will be expressed as the disjoint union of two subgraphs: the 
\emph{remainder} $\Delta_i$ and the \emph{reservoir} $\Omega_i$.

Each of these graphs are thought of as metric graphs, whose edges have lengths 
equal to the length of the words that label them.  The {\em mass} of a metric 
graph $\Gamma$ is its total length, denoted by $m(\Gamma)$. The {\em type} of a 
graph refers to the collection of edge labels (which are reduced words in $F_k$) 
associated to each edge. A \emph{distribution} on a certain set of types of graphs 
is a map that assigns a non-negative number to each type; it is \emph{integral} if 
it assigns an integer to each type. We use this terminology without comment in the sequel.

The following properties will remain true at every stage of our construction.  
The branch vertices of $\Lambda_i$ (those of valence greater than two) have valence four.  
The length of each edge will always be at least $\lambda$.  
We will be careful to ensure that any branch vertex in the interior of the glued 
part $\Gamma_i$ is \emph{good} in the sense of Section \ref{section:spines}.  
Vertices in the intersection of the glued and unglued parts, $\Gamma_i\cap\Upsilon_i$, 
will always be branch vertices, and will be such that one adjacent edge is in $\Gamma_i$, 
and the remaining three adjacent edges in $\Upsilon_i$ have distinct labels. 

In particular, we start with $\Lambda_0=\Upsilon_0$ and $\Gamma_0=\varnothing$.  
At the last stage of our construction we will have $\Lambda_8$ constructed from 324 copies 
of $\Lambda_0$ (or 2,916 if $k=2$), which is completely glued up; that is, $\Gamma_8=\Lambda_8$ and $\Upsilon_8=\varnothing$. Finally, the modification in Section~\ref{subsection:coorientation} doubles the mass of $\Lambda_8$ in order to ensure that the {\em co-orientation condition} is satisfied. Taking $Z$ to be the result of this construction and $f$ to be the quotient map $L\to Z$ proves the theorem.

\subsection{Football bubbles}\label{subsection:football_bubbles}

We will regard $k$ and $\lambda$ as constants.  The first step of the construction is to pick some very big constant $N \gg \lambda$ where still $T\gg N$
(we will explain in the sequel how to choose $T$ and $N$ big enough) so that $N$ is odd.

Let $s$ be the remainder when $|r|-3\lambda$ is divided by $3(N+1)\lambda$. 
By pseudorandomness, we may find three subsegments in $r$ of reduced form
$$a_1xa_2,\quad b_1xb_2,\quad c_1xc_2$$
where $a_i,b_i,c_i$ are single edges, such that the labels $a_1,b_1,c_1$ are all distinct and $a_2,b_2,c_2$ are all distinct, and the length of $x$ is $s$. We take three copies of $\Lambda_0$, fix one of the above subsegments in each copy, and glue the parts of these subsegments labeled $x$  together to obtain $\Lambda_1$.  Note that the requirement that the labels $a_i,b_i,c_i$ are distinct ensures that $\Lambda_1$ remains folded.  We summarize this in the following lemma.

\begin{lemma}
After gluing three copies of $\Lambda_0$ along subsegments of length $s$ we obtain $\Lambda_1$, with the property that the length of each edge of the unglued subgraph $\Upsilon_1$ is congruent to $3\lambda$ modulo $3(N+1)\lambda$. The glued subgraph $\Gamma_1$ is a segment of length $s$.
\end{lemma}

We next decompose the unglued subgraph $\Upsilon_1$ into disjoint segments of length $3N\lambda$ separated by segments of length $3\lambda$. Call the segments of length $3N\lambda$ {\em long strips} and the segments of length $3\lambda$ \emph{short strips}.   Now further decompose each long strip into alternating segments of length $3\lambda$; we call the odd numbered segments {\em sticky} and the even numbered segments {\em free}.

We will usually denote a long strip by
$$x_1a_2x_3\ldots x_N$$
where the $x_i$ are sticky, the $a_i$ (or $b_i$ etc) are free, and all are of length $3\lambda$.  When we also need to include the neighboring short strips, we will usually extend this notation to
$$a_0x_1a_2x_3\ldots x_Na_{N+1}$$
where $a_0$ and $a_{N+1}$ (or $b_0,b_{N+1}$ etc) denote the neighboring short strips.

\begin{definition}\label{definition:compatible}
Three long strips are {\em compatible} if they (and their adjoining short strips) are of the form
$$a_0 x_1 a_2 x_3 \cdots a_{N+1}, \quad b_0 x_1 b_2 x_3 \cdots b_{N+1}, \quad c_0 x_1 c_2 x_3 \cdots c_{N+1}$$
(i.e.\/ if their sticky segments agree) and if for even $i$ (i.e.\/ for the free segments) the
letters adjacent to each $x_{i+1}$ or $x_{i-1}$ disagree.
\end{definition}

A compatible triple of long strips can be {\em bunched} --- i.e.\/ the sticky segments can be glued
together in threes, creating $(N-1)/2$ {\em football bubbles}, each football consisting of the three
segments $a_i, b_i, c_i$ (for some $i$) arranged as the edges of a theta graph. See Figure~\ref{bubbles}.

\begin{figure}[htpb]
\labellist
\small\hair 2pt
%\pinlabel $\text{Type A}$ at -200 0
\endlabellist
\centering
\includegraphics[scale=0.7]{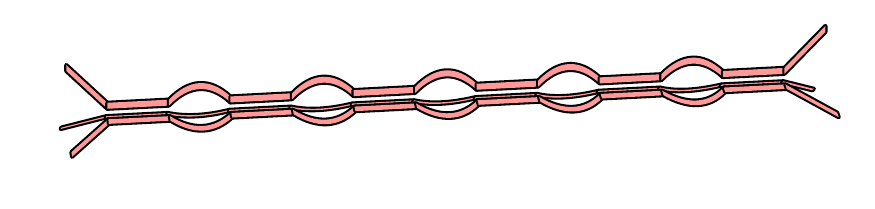}
\caption{Bunching three long strips to create football bubbles}\label{bubbles}
\end{figure}

By pseudorandomness, a proportion of approximately $(1-\epsilon)$ of the long strips in $\Gamma_1$ can be partitioned into compatible triples.  Then each compatible triple can be bunched, creating a {\em reservoir} of footballs (i.e.\ the theta graphs appearing as bubbles) and a {\em remainder}, consisting of the union of the unglued pieces (except for the footballs).  In what follows, we will denote the remainder in $\Upsilon_i$ by $\Delta_i$ and the reservoir by $\Omega_i$.
 
We summarize this observation in the next lemma.  By an \emph{extended long strip}, we mean a long strip, together with half the adjacent short strips.

\begin{lemma}
Suppose that $T\geq 3(N+1)\lambda$, and that $N$ is sufficiently large and $\epsilon$ is sufficiently small.  After bunching compatible triples in the unglued graph $\Upsilon_1$ we obtain the partially glued graph $\Lambda_2$ with the following properties.
\begin{enumerate}
\item The total mass of the unglued subgraph $\Upsilon_2$ satisfies
\begin{equation*}\label{eqn:U2}
\frac{m(\Upsilon_2)}{3m(\Lambda_0)}\leq \frac{1}{2}+O(c^N)\epsilon
\end{equation*}
for a constant $c=c(k,\lambda)$.

\item We can decompose the unglued subgraph $\Upsilon_2$ as a disjoint union $\Delta_2\sqcup\Omega_2$.   The reservoir $\Omega_2$ is a disjoint union of bubbles.
\item The mass of the remainder satisfies
\begin{equation*}\label{edn:D2}
\frac{m(\Delta_2)}{3m(\Lambda_0)}=O(1/N)
\end{equation*}
as long as $\epsilon<O(c^{-N})$.
\item The distribution of the types of bubbles in the reservoir is within $O(\epsilon)$ of a constant distribution (independent of $N$ and $\epsilon$).
\end{enumerate}
\end{lemma}

%After bunching up compatible triples, we estimate the mass of the unbunched part in units where the ``total mass'' of the unglued chain is 1.  Since we are only concerned with order of magnitude comparisons, we do not labor to compute exact values (whose complicated expression would only obscure the ideas); this should cause no confusion in what follows.
\begin{proof}
The number of types of extended long strips is $O(c^N)$ for some constant $c=c(k,\lambda)$.  The unglued subgraph $\Upsilon_1$ is still $(T,\epsilon)$-pseudorandom, containing a union of
\[
\frac{m(\Upsilon_1)}{3(N+1)\lambda}
\]
extended long strips.  By pseudorandomness, we may restrict to a subset $\Upsilon'_1\subseteq \Upsilon_1$ of mass at least $(1-O(c^N)\epsilon)m(\Upsilon_1)$ so that the types of extended long strips in $\Upsilon'_1$ are \emph{exactly} uniformly distributed.  (Here we use that $T\geq 3(N+1)\lambda$.)   
We then randomly choose a partition into compatible triples, and perform bunching, to
produce $\Lambda_2$.   

We estimate the mass of the unglued subgraph as follows.
\begin{eqnarray*}
m(\Upsilon_2)&=&m(\Upsilon_1\smallsetminus\Upsilon'_1)+\frac{m(\Upsilon'_1)}{2}\\&\leq& O(c^N)\epsilon m(\Upsilon_1)+\frac{(1-O(c^N)\epsilon)}{2}m(\Upsilon_1)\\
&=&\frac{(1+O(c^N)\epsilon)}{2}m(\Upsilon_1)
\end{eqnarray*}
and (\ref{eqn:U2}) follows immediately.

As described above, the unglued subgraph $\Upsilon_2$ is naturally a disjoint union of the remainder $\Delta_2$ and the reservoir $\Omega_2$.  The remainder consists, by definition, of the union of $\Upsilon_1\smallsetminus\Upsilon'_1$ and the short strips in $\Upsilon'_1$ (after gluing).  Hence
\[
m(\Delta_2)<O(c^N)\epsilon m(\Upsilon_1)+\frac{1-O(c^N)\epsilon}{N+1}m(\Upsilon_1)= \frac{(1+O(c^N)\epsilon)}{N+1}m(\Upsilon_1)
\]
which is $O(1/N)$ as long as $\epsilon<O(c^{-N})$.

Bunching uniformly distributed long strips at random induces a fixed distribution  on subgraphs of bounded size. This justifies the final assertion.
\end{proof}

\subsection{Super-compatible long strips}\label{subsection:supercompatible}

At this point, we have bunched all but $\epsilon$ of the long strips into triples.    
In particular, for sufficiently small $\epsilon$, the number of bunched triples is far 
larger than the number of unbunched long strips.  We will now argue that we may, in fact, 
adjust the construction so that {\em every} long strip is bunched.  The advantage of this
is that after this step, the unglued part consists entirely of the reservoir (a union of 
football bubbles) and a remainder consisting of a trivalent graph in which {\em every} edge
has length {\em exactly} $3\lambda$. The key to this operation is the 
idea a \emph{super-compatible} 4-tuple of long segments.

\begin{definition}\label{definition:super-compatible}
Four long strips are {\em super-compatible} if they are of the form
$$a_0 x_1 a_2 x_3 \cdots a_{N+1}, \quad b_0 x_1 b_2 x_3 \cdots b_{N+1}, 
\quad c_0 x_1 c_2 x_3 \cdots c_{N+1}, \quad d_0 x_1 d_2 x_3 \cdots d_{N+1}$$
(i.e.\/ if their sticky segments agree) and if for even $i$ (i.e.\/ for the free segments) 
the initial and terminal letters of $a_i,b_i,c_i,d_i$ disagree.  Alternatively, such a 4-tuple 
is super-compatible if every sub-triple is compatible.
\end{definition}

\begin{remark}
Notice that the existence of super-compatible 4-tuples depends on rank $k\ge 3$. An 
alternative method to eliminate unbunched long strips in rank 2 is given in \S~\ref{rank_2_subsection}.
\end{remark}

\begin{lemma}
Let $\Lambda_2$ be as above. As long as $k\geq 3$ and $\epsilon$ is sufficiently small (depending on $N$), one can injectively assign to each long strip $S_0$ in the remainder $\Delta_2$ a bunched triple $(S_1,S_2,S_3)$ in $\Lambda_2$ such that the quartet $(S_0,S_1,S_2,S_3)$ is super-compatible. 
\end{lemma}
\begin{proof}
Let $S_0$ be an unglued long strip.  As long as $k\geq 3$, it is clear that there is at least one type of bunched triple $(S_1,S_2,S_3)$ such that $(S_0,S_1,S_2,S_3)$ is super-compatible. The total number of types of bunched triples is $O(D^N)$ for some $D=D(k,\lambda)$.   Therefore, the proportion of bunched triples that are super-compatible with $S_0$ is at least $1/O(D^N)-O(\epsilon)$.  On the other hand, the proportion of long strips that are unbunched is $O(\epsilon)$.  Therefore, we can choose a bunched triple for every unbunched long strip as long as $O(\epsilon)<1/O(D^N)$.
\end{proof}

By the previous lemma, we may assign to each unglued long strip $S_0$ in the remainder 
$\Delta_2$ a glued triple $(S_1,S_2,S_3)$ such that $(S_0,S_1,S_2,S_3)$ is super-compatible.  
We may re-bunch these into the four possible compatible triples that are subsets of our 
super-compatible 4-tuple, viz:
\[
(S_0,S_1,S_2),~(S_0,S_1,S_3),~(S_0,S_2,S_3),~(S_1,S_2,S_3)
\]
The result of performing this operation on every long strip in the remainder $\Delta_2$ yields the new partially glued graph $\Lambda_3$.  

\begin{lemma}
Take 3 copies of $\Lambda_2$ as above, and suppose that $N$ is sufficiently large and $\epsilon$ is sufficiently small.  After choosing super-compatible triples and re-bunching so that every long strip is bunched, we produce the partially glued graph $\Lambda_3$ with the following properties.
\begin{enumerate}
\item The partially glued graph $\Lambda_3$ consists of a glued subgraph $\Gamma_3$, a reservoir $\Omega_3$, and a remainder $\Delta_3$.
\item The reservoir $\Omega_3$ consists of football bubbles and its mass is bounded below by
\begin{equation*}\label{eqn:O3}
\frac{m(\Omega_3)}{m(\Lambda_0)}\geq O(1)
\end{equation*}
\item The remainder $\Delta_3$ is a trivalent graph with each edge of length $3\lambda$, and its mass is bounded above by
\begin{equation*}\label{eqn:D3}
\frac{m(\Delta_3)}{m(\Lambda_0)}\leq 1/O(N)
\end{equation*}
\item Let $\rho$ be the uniform distribution on types of bubbles.  Then for any type $B$ of bubble, the proportion of bubbles in the reservoir of type $B$ is within $O(c^N)\epsilon$ of $\rho$.
\end{enumerate}
\end{lemma}
\begin{proof}
Consider $\Upsilon_1$, the unglued subgraph of $\Lambda_1$.  Then $\Upsilon_1$ is the union of three arcs, of total mass $3C(N+1)\lambda+9\lambda$, where $C$ is the number of long strips in $\Upsilon_1$.  Recall that $\Lambda_3$ is constructed from three copies of $\Lambda_1$, and that for each long strip, one short strip goes into the remainder, and $(N-1)/2$ go into the reservoir.  Therefore, we have that
\[
\frac{m(\Omega_3)}{3m(\Upsilon_1)-9\lambda}=\frac{N-1}{2N+2}
\]
and 
\[
\frac{m(\Delta_3)-9\lambda}{3m(\Upsilon_1)-9\lambda}=\frac{1}{N+1}
\]
Since $m(\Upsilon_1)\geq 3m(\Lambda_0)-O(N)$, the estimates (\ref{eqn:O3}) and (\ref{eqn:D3}) follow.

Finally, we need to check that the re-gluing only has a small effect on the distribution of bubble types.  Recall that the unglued long strips in $\Upsilon_2$ are precisely the images of the unglued long strips in $\Upsilon_1\smallsetminus\Upsilon'_1$, which is of mass at most $O(c^N)\epsilon m(\Upsilon_1)$.

Taking three copies of each of these and three copies of a super-compatible triple, re-bunching produces four new bunched triples.  In particular, for each three unbunched long strips, we destroy $(N-1)/2$ bubbles and replace them with $2(N-1)$ new bubbles of different types.  The proportion of unglued long strips was at most $O(c^N)\epsilon$.  It follows that the proportional distribution of each bubble type was altered by at most $O(c^N)\epsilon$, so the final estimate follows.
\end{proof}

\subsection{Inner and outer reservoirs and slack}

As their name indicates, the bubbles in the reservoir will be held in reserve until a later
stage of the construction to glue up the remainder. Some intermediate operations on the
remainder have ``boundary effects'', which might disturb the neighboring bubbles in the reservoir
in a predictable way. So it is important to insulate the remainder with a collar of bubbles
which we do not disturb accidentally in subsequent operations.

Fix a constant $0<\theta_n < 1$ and divide each long strip in $\Lambda_n$ into two parts:
and {\em inner reservoir}, consisting of an innermost sequence of consecutive bubbles of length
$1-\theta_n$ times the length of the long strip, 
and an {\em outer reservoir}, consisting of two outer sequences of consecutive
bubbles of length $\theta_n/2$. The number $\theta_n$ is called the {\em slack}.
Boundary effects associated to each step that we perform will use up bubbles from 
the outer reservoir and at most halve the slack. Since the number of steps we perform
is uniformly bounded, it follows that --- provided $N$ is sufficiently large --- 
even at the end of the construction we will still have a significant outer reservoir 
with which to work.

\subsection{Adjusting the distribution}\label{subsection:distribution}

After collecting long strips into compatible triples, the collection of football bubbles in the
reservoir is ``almost equidistributed'', in the sense that the mass of any two different types of
football is almost equal (up to an additive error of order $\epsilon$). However, it is useful
to be able to adjust the pattern of gluing in order to make the distribution of football bubbles
conform to some other specified distribution (again, up to an additive error of order $\epsilon$).

This operation has an unpredictable effect on the remainder, transforming it into some new 3-valent
graph (of some possibly very different combinatorial type); however, it preserves the essential
features of the remainder that are known to hold at this stage of the construction: every edge
of the remainder (after the operation) has length exactly $3\lambda$; and the total mass of the
remainder before and after the operation is unchanged (so that it is still very small compared to
the mass of any given football type).

Let $\mu$ be a probability measure on the set of all football types, with full support  --- i.e.\/ 
so that $\mu$ is strictly positive on every football type.  (In the sequel, $\mu$ will be the \emph{cube distribution} described below, but that is not important at this stage.) Suppose we have three long strips of the
form 
$$a_0x_1a_2x_3\cdots a_{N+1}, \quad b_0x_1b_2x_3\cdots b_{N+1}, \quad c_0x_1c_2x_3\cdots c_{N+1}$$
so that the result of the gluing produces $(N+1)/2$ bunches each with the label $x_i$ (for $i$ odd), and
$(N-1)/2$ football bubbles each with the label $(a_i,b_i,c_i)$ (for $i$ even). We think of the $(a_i,b_i,c_i)$ as {\em unordered} triples --- i.e.\/ we only think of the underlying football as an abstract graph with edge labels up to isomorphism. A given sequence of labels $x_i$ and (unordered!) football types $(a_i,b_i,c_i)$ might arise from three long strips $s,t,u$ in $6^{(N+3)/2}$ ways, since there are 6 ways to order each triple $a_i,b_i,c_i$. 

Let $\rho$ denote the uniform probability measure on football types, and let $\mu'$ be chosen to be a multiple of $\mu$ such that $\rho>\mu'$ for all types. Fix $\theta_4>0$ such that $\min \mu'/\rho>\theta_4$.  As the notation hints, $\theta_4$ will turn out to be the slack in the partially glued graph $\Lambda_4$, and we accordingly partition each long strip of $\Lambda_3$ into an inner reservoir, of proportional length $(1-\theta_4)$, and an outer reservoir consisting of two strips of proportional length $\theta_4/2$.

We put all possible triples of long strips labeled as above into a bucket.   Next, color each even index $i$ in $[N\theta_4/2,N(1-\theta_4/2)]$ (i.e.\ those contained in the inner reservoir) black with probability
\[
\frac{(\rho-\mu'(i))}{(1-\theta_4)\rho}
\]
and color all the remaining indices white, where $\mu'(i)$ is short for the $\mu'$ measure of the football type $(a_i,b_i,c_i)$ 
(note that our choice of $\theta_4$ guarantees that the assigned probabilities are never 
greater than 1).

Now pull apart all the triples of long strips, and match them into new triples $s,t,u$ according to the 
following rule: if a given index $i$ is white, the corresponding labels $s_i,t_i,u_i$ should all
be {\em different}, and equal to $a_i,b_i,c_i$ (in some order); if a given index $i$ is black, the
corresponding labels $s_i,t_i,u_i$ should all be the {\em same}, and equal to exactly
one of $a_i,b_i,c_i$. Then we can glue up $s,t,u$ to produce footballs $(a_i,b_i,c_i)$ 
exactly for the white labels, and treat the black labels as part of the neighboring sticky segments,
so that they are entirely glued up.

We do this operation for each bucket (i.e.\ for each collection of triples with a given sequence
of sticky types $x_i$ and football types $(a_i,b_i,c_i)$). The net effect is to eliminate  a fraction of approximately $(\rho-\mu'(i))/\rho$ of the footballs with label $i$; thus, at the end of this operation, the distribution of footballs is proportional to $\mu$, with error of order $\epsilon$.

Although this adjustment operation can achieve any desired distribution $\mu$,
in practice we will set $\mu$ equal to the \emph{cube distribution}, to be described in the sequel.
In any case, for a fixed choice of distribution $\mu$, the slack $\theta_4$ only depends on $k$ and $\lambda$, and therefore can be treated as a constant.

We summarize this in the following lemma.

\begin{lemma}\label{lem: Adjusting the distribution}
Let $\mu$ be a probability distribution on the set of types of football bubbles and let $N$ be sufficiently large.  The adjustment described above transforms $\Lambda_3$ into a new partially glued graph $\Lambda_4$ with the following properties.
\begin{enumerate}
\item The mass of the reservoir $\Omega_4$ is bounded below by a constant depending only on $k,\lambda$ and $\mu$.
\[
\frac{m(\Omega_4)}{m(\Lambda_0)}\geq O(1)
\]
\item The mass of the remainder $\Delta_4$ is bounded above by
\[
\frac{m(\Delta_4)}{9m(\Lambda_0)}\leq 1/N
\]
\item The distribution of football types in the reservoir is proportional to $\mu$, with error $O(c^N)\epsilon$.
\item The slack $\theta_4$ is a constant.
\end{enumerate}
\end{lemma}
\begin{proof}
As noted above, this operation may completely change the combinatorial type of the remainder, but leaves invariant its total mass, and the fact that it is a 3-valent graph with edges of mass $3\lambda$.  In particular,
\[
\frac{m(\Delta_4)}{9m(\Lambda_0)}\approx \frac{1}{N+1}
\]
and so $m(\Delta_4)/9m(\Lambda_0)\leq 1/N$ for sufficiently large $N$.  This proves item 2.

We lose a fraction of the reservoir---those indices colored black.  An index $i$ between $N\theta_4/2$ and $N(1-\theta_4/2)$ is colored black with probability $(\rho-\mu'(i))/\rho$.  Therefore, the proportion of bubbles colored white is bounded below by $\min_{i}\mu'(i)/\rho$, and so the mass of the reservoir is bounded below by
\[
\frac{m(\Omega_4)}{3m(\Upsilon_1)-9\lambda}\geq \left(\min_{i}\mu'(i)/\rho\right)\frac{N-1}{2N+2}
\]
Since $\mu'$ is a distribution on types of bubbles, depending only on $\lambda$ and $k$, the first item holds as long as $N$ is sufficiently large.

We next estimate the distributions of the types of bubbles.  Before adjustment, the proportion of each type of bubble in the reservoir $\Omega_3$ was within $O(c^N)\epsilon$ of the uniform distribution.  These can be taken to be uniformly distributed between the inner and outer reservoirs.  Therefore, after adjustment, the new distribution $\nu$ on bubble types satisfies
\[
|\nu-\mu'|\leq (\mu'/\rho)O(c^N)\epsilon
\]
and so, as before, since $\mu'(i)$ is bounded above in terms of $\lambda$ and $k$, the third assertion follows.

The final assertion about the slack is immediate from the construction.
\end{proof}

\subsection{Tearing up the remainder}\label{subsection:tear}

At this stage the remainder consists of a 3-valent graph in which every edge has length exactly $3\lambda$. The total mass of the remainder is very small compared to the mass of the reservoir, but it is large compared to the size of a single long strip. Furthermore, there is no {\it a priori} bound on the combinatorial complexity of a component of the remainder. 

We explain how to modify the gluing by a certain local move called a {\em tear}\footnote{``tear'' in the sense of: ``There were tears in her big brown overcoat''}, which (inductively) reduces the combinatorial complexity of the remainder (which {\it a priori} is arbitrarily complicated) until it consists of a disjoint collection of simple pieces. These pieces come in three kinds:
\begin{enumerate}
\item{football bubbles;}
\item{bizenes: these are graphs with 6 edges and 4 vertices, obtained from a square
by doubling two (non-adjacent) edges; and}
\item{bicrowns: these are complete bipartite graphs $K_{3,3}$.}
\end{enumerate}
The bubbles, bizenes and bicrowns all have edges of length exactly $3\lambda$. They are depicted in
Figure~\ref{bbb}. Note that bizenes doubly cover footballs and bicrowns triply cover footballs.  If the labels on a bizene or bicrown happen to be pulled back from the labels on a football bubble via the covering map, then we say that the bizene or bicrown is \emph{of covering type}.

\begin{figure}[htpb]
\labellist
\small\hair 2pt
%\pinlabel $\text{Type A}$ at -200 0
\endlabellist
\centering
\includegraphics[scale=0.5]{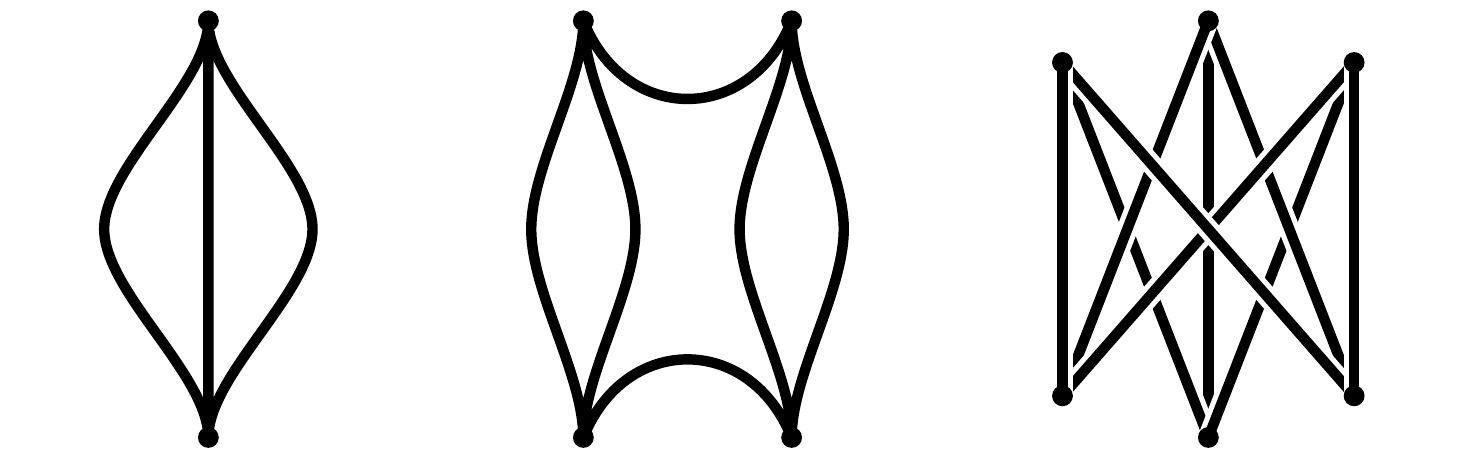}
\caption{Football bubbles, bizenes, and bicrowns}\label{bbb}
\end{figure}

\iffalse

We think of the remainder $\Delta_4$ as a trivalent graph with edges of length $3\lambda$. Each vertex of $\Delta_4$ is the end of a long segment which  alternates between bunched triples of
sticky segments, and football bubbles. As we make adjustments to the gluing we will
produce new trivalent graphs as the remainder, but what is attached to each vertex
will no longer be a long segment. We account for this in the following way: for each
vertex of $\Delta_4$, we define the {\em slack} to be the biggest integer $i$ which is equal
to half the length of the sequence of alternating bunched triples and football bubbles in attached
to $\Delta_4$ at that vertex; thus initially, the slack of every vertex of $\Delta_4$ is of order $N$.

The slack only becomes significant when it is smaller than a specific constant, say $4$, since
each tear move definitely reduces slack, and slack values must always stay positive. 
We will define an operation which
takes as input a pair of vertices of $\Lambda_4$, and the result of this operation
transforms the remainder in a predictable way, while dividing the slack at these 
vertices approximately in half. So if we are careful to apply our move only boundedly 
many times to each vertex of $\Lambda_4$, we always have ample slack. 

\fi

\begin{figure}[htpb]
\labellist
\small\hair 2pt
%\pinlabel $\longrightarrow$ at 50 50
\endlabellist
\centering
\includegraphics[scale=0.5]{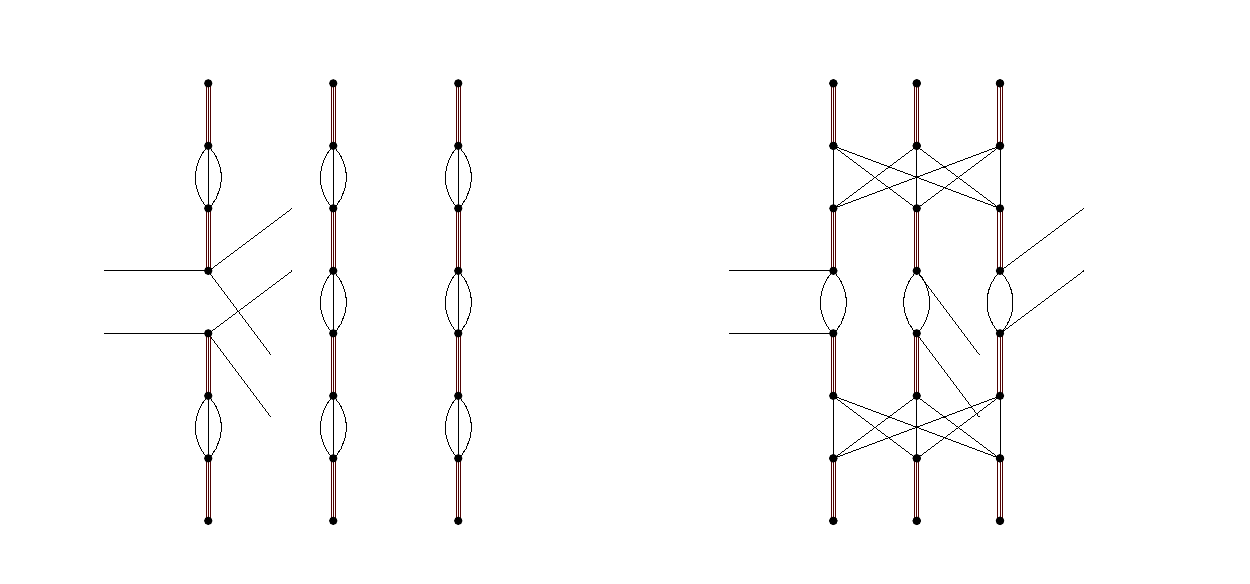}
\caption{Tearing at $p$ and $p'$. The ``before'' picture is on the left, and the ``after'' picture on the
right. The graph $\Lambda_4\cup\Lambda'_4$ is transformed by being disconnected
at $p$ and $p'$ and having three new bigons inserted which connect the new six 1-valent ends 
of $\Lambda_4 \cup\Lambda'_4 - p - p'$ in pairs.  The move also destroys 8 footballs, and creates two bicrowns.}\label{tearing}
\end{figure}

We now describe the operation of {\em tearing}.  Take two copies of $\Lambda_4$, denoted by $\Lambda_4$ and $\Lambda'_4$ (with vertices, edges and subgraphs of $\Lambda'_4$ denoted with primes in the obvious manner) and let $p$ and $p'$ be branch vertices of $\Delta_4$ and $\Delta'_4$ respectively.  These vertices are the ends of disjoint sequences of alternating bunched triples and football bubbles.
%the number of which is equal to the slack values $i$ and $i'$ at $p$ and $p'$.
The tearing operation uses up two (appropriately labeled) sequences of alternating bunched
triples and football bubbles, each of length 3. The precise definition of this operation is best
given by example, and is illustrated in Figure~\ref{tearing}.

On the left of the figure we have the vertices $p$ and $p'$ of $\Lambda_4$ and $\Lambda'_4$, 
together with two strings of $3$ bubbles. These strings are pulled apart and reglued according to the combinatorics indicated in the figure. Thus, the labels on the bunched triples at each horizontal level should all agree, and the labels on the footballs should
be such that the result of the gluing is still folded. The existence of strings of football
bubbles with these properties is guaranteed by pseudorandomness and the definition of the long
strips. 

Thus the operation of tearing uses up 8 footballs (as in the figure), and it has several effects on the remainder. 
First, $\Lambda_4\cup\Lambda'_4$ is pulled apart at $p$ and $p'$, producing six new vertices $p_1,p_2,p_3$ 
and $p_1',p_2',p_3'$, and adding two new edges $x_i$ and $x_i'$ (from the footballs 
separating the adjacent bunched triples) joining $p_i$ to $p_i'$. Second, two new bicrowns are
created, assembled from the pieces of three identically labeled footballs. 
Third, the slack at $p$ and $p'$ is reduced to approximately half of its previous value. If the strings of three football segments are taken from the inner half of the outer reservoir, it will reduce the slack at the vertices of $\Lambda_4\cup\Lambda'_4$ at the end of these strips; but the size of the slack at these vertices will stay large.

\begin{lemma}\label{lem: Tearing}
Suppose that $N$ is sufficiently large, $\epsilon$ is sufficiently small and the partially glued graph $\Lambda_4$ is as above.  By applying tearing operations to 2 copies of $\Lambda_4$ we may build a partially glued graph $\Lambda_5$ with the following properties.
\begin{enumerate}
\item The mass of the reservoir $\Omega_5$ is bounded below by a constant depending only on $k,\lambda$ and $\mu$.
\[
\frac{m(\Omega_5)}{m(\Lambda_0)}\geq O(1)
\]
\item The remainder $\Delta_5$ is a disjoint union of bizenes, and bicrowns of covering type, with mass bounded above by
\[
\frac{m(\Delta_5)}{m(\Lambda_0)}\leq O(1/N)
\]
\item The distribution of football bubbles of each type is within $O(1/N)$ of the distribution $\mu'$ (proportional to $\mu$).
\item The slack satisfies $\theta_5\geq \theta_4/3$.
\end{enumerate}
\end{lemma}
\begin{proof}
As described above, we construct $\Lambda_5$ from two copies of $\Lambda_4$: let us denote them by $\Lambda_4$ and $\Lambda'_4$, and likewise denote subgraphs, vertices and edges with primes as appropriate.  At each branch vertex $p$ of $\Delta_4$ we have three unglued (elementary) edges with labels $a,b,c$ (pointing away from $p$, say), and one glued edge with label $d$ (also pointing away from $p$).  Denote by $e_x$ the short strip in $\Delta_4$ incident at $p$ with outgoing label $x$.  Let $p'$ be the corresponding vertex of $\Delta'_4$, which is of course locally isomorphic to $p$.

For each such pair of vertices $p$ of $\Delta_4$ and $p'$ of $\Delta'_4$, we choose a pair of bubbles so that we may perform a tear move at $p$ and $p'$.  In order to do this, we must choose a pair of bubbles $B_1$, $B_2$, with certain constraints on the labelings at their branch vertices.  We next describe one feasible set of constraints that enables the tearing operation to be performed (there will typically be other possible configurations).

Necessarily, at each branch vertex of $B_1$ and $B_2$, we need the incident glued (elementary) edge to have (outgoing) label $d$.  We will also require that each bubble $B_i$ is a union of three short strips $s^i_a,s^i_b,s^i_c$, with the property that at each branch vertex of $B_i$ the outgoing label on the short strip $s^i_x$ is equal to $x$.  Since there are only a finite number of possible local labelings at the branch vertices, and since each type of bubble occurs with roughly equal distribution, there are many bubbles satisfying this condition.

Later in the argument, it will also prove necessary that the strips $s^i_\bullet$ satisfy certain other constraints (see Lemma \ref{lem: Gluing remainder bizenes} below).  For the moment, it suffices that these constraints are mild enough to guarantee the existence of the bubbles $B_i$.

Given bubbles $B_1$, $B_2$ for a vertex $p$ of $\Delta_4$, we can perform the tearing operation, in such a way that after tearing, $e_a$ and $e'_a$ adjoin $s^1_b$ and $s^2_c$, $e_b$ and $e'_b$ adjoin $s^1_c$ and $s^2_a$, and $e_c$ and $e'_c$ adjoin $s^1_a$ and $s^2_b$.   Note that the resulting graph remains folded, and that the remainder $\Delta_4$ has been replaced by a union of bizenes.

Therefore, in order to perform the tearing operation, we need to find $V$ pairs of bubbles $B_1,B_2$ as above, where $V$ is the number of vertices of $\Delta_4$.   To do this, we divide all football bubbles in the undisturbed segments of the long strips into consecutive runs of seven.   For each vertex of the remainder $\Delta_4$, we need to choose two such runs of a specified type.   We furthermore insist on choosing these runs of seven from within the `innermost' part of the `outer' reservoir.

The number $V$ is equal to $2m(\Delta_4)/9\lambda\leq m(L)O(1/N)$. By pseudorandomness, the number of runs of seven bubbles of a fixed type is bounded below by
\[
\theta_4O(m(\Omega_4))\geq \theta_4O(m(L))
\]
(using that $\epsilon$ is sufficiently small).   Half of these runs of seven come from within the innermost part of the outer reservoir.  Therefore, as long as $N$ is large enough, we can always choose two suitable runs of seven bubbles for each vertex of $\Delta_4$, as required.

Since $V/m(L)$ is bounded above by a constant (depending on $k,\lambda$ and $\mu$) divided by $N$, the distribution of each bubble type has only been changed by $O(1/N)$.

After performing tears in this way, we obtain a new partially glued graph $\Lambda_5$, with the additional property that the remainder $\Delta_5$ is a disjoint union of bizenes and bicrowns (the latter of covering type). The mass of the remainder is still bounded above by
\[
m(\Delta_5)/m(L)\leq 12/N
\]
since three half-edges of $\Delta_4$ are replaced by 36 half-edges of $\Delta_5$, as shown in Figure \ref{tearing}.

Since, by construction, we only used bubbles in the tearing operation which came from the innermost half of the outer reservoir together with one bubble from the outermost part of the outer reservoir, the slack $\theta_5$ is no smaller than
\[
\theta_4(1/2-O(1/N))\geq \theta_4/3
\]
as claimed.
\end{proof}

\subsection{Adjusting football inventory with trades}\label{subsection:trades}

It will be necessary at a later stage of the argument to adjust the numbers of football pieces
of each kind, so that the reservoir itself can be entirely glued up. At this stage and subsequent
stages we must be careful to consider not just the combinatorial graph of our pieces, but also 
their {\em type} --- i.e.\/ their edge labels.

\iffalse
Remember that a bicrown triply covers a football;
let $\pi:b \to f$ denote the covering projection. We say that a (labeled) bicrown is {\em of covering type}
if the edge labels on $b$ are compatible with edge labels pulled back from some $f$ under $\pi$.
\fi

We now describe a move called a {\em trade} which has the following twofold effect:
\begin{enumerate}
\item{it reduces the number of footballs of a specified type by 3; and}
\item{it transforms four sets of 3 footballs, each of a specified type, into four bicrowns
each with the associated covering type.}
\end{enumerate}

Moreover, {\em unlike} the operation described in \S~\ref{subsection:distribution}, the trade
operation has {\em no effect on the remainder}. Thus, the trade moves can be performed {\em after}
the tear moves, to correct small errors in the distribution of football types, adjusting this
distribution to be {\em exactly} as desired.

\iffalse
In \S~\ref{subsection:cube_move} we will explain how four footballs with suitable types may be
glued up by arranging them suitably on the 1-skeleton of a cube (informally: this is the
application of a {\em cube move}). Taking a triple cover of this
arrangement gives an arrangement of four bicrowns; thus four bicrowns of covering type may be
entirely glued up if the corresponding football types they cover may be entirely glued up by
a cube move. 
\fi

The trade move is illustrated in Figure~\ref{trade}. We start with three strings of five
footballs, each string consisting of the same sequence of five football types in the same
orders. We also assume the labels on the three sets of four intermediate sticky segments agree.
We pull apart the sticky segments and reglue them in the pattern indicated in the figure, in
such a way that four sets of three footballs are replaced with bicrowns. If the three middle
footballs are of type $(a,b,c)$ then after regluing we can assume that the $a$ edges are all
together, and similarly for the $b$ and $c$ edges; thus these triples of edges may by glued
up, eliminating the three footballs.

\begin{figure}[htpb]
\labellist
\small\hair 2pt
%\pinlabel $\longrightarrow$ at 50 50
\endlabellist
\centering
\includegraphics[scale=0.4]{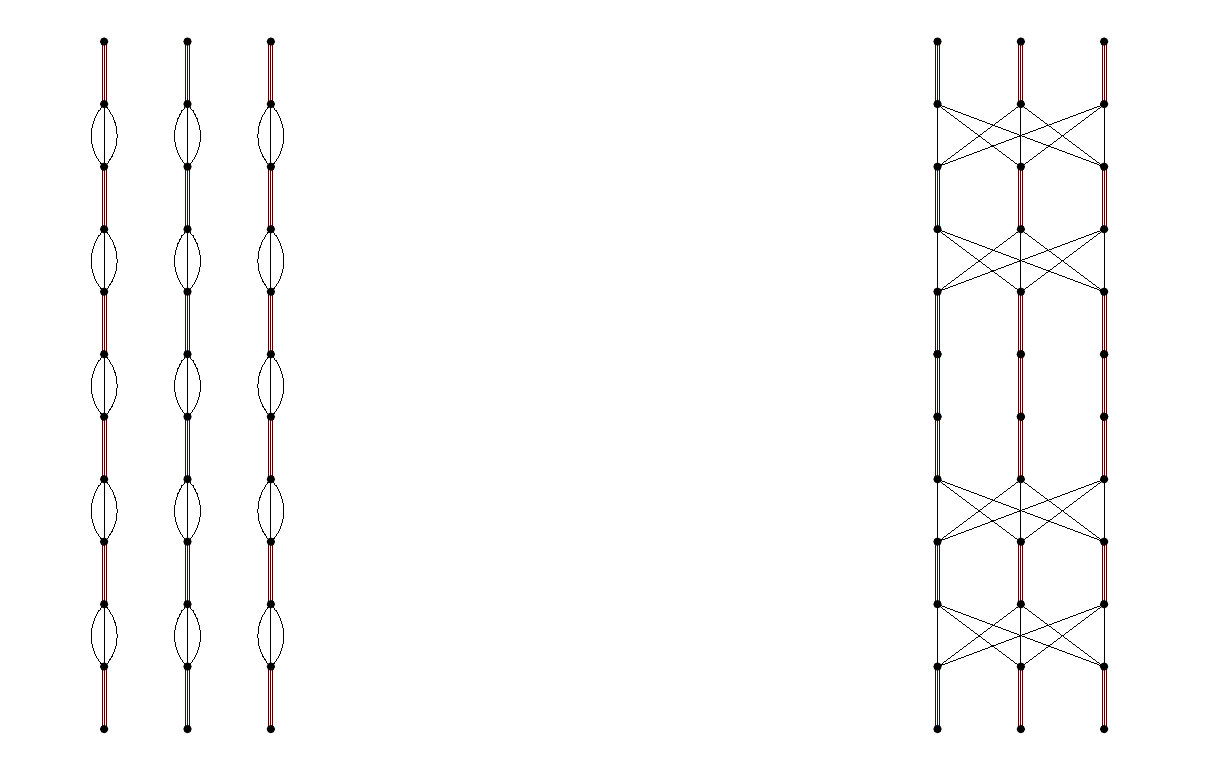}
\caption{The trade move; the ``before'' picture is on the left, and the ``after'' picture is on the
right. This move eliminates 3 footballs of a given type, and transforms four sets of
3 footballs with a given type into four bicrowns of the given covering type.}\label{trade}
\end{figure}

To summarize this succinctly, we introduce some notation.  For $B$ a type of football bubble, we denote by $\widetilde{B}$ the corresponding covering type of bicrown.  For $\mu$ a distribution on football bubbles and bicrowns of covering type, we denote by $\tilde{\mu}$ the following distribution on football bubbles.
\[
\tilde{\mu}(B)=\mu(B)+3\mu(\widetilde{B})
\] 

\begin{lemma}\label{lemma:trades}
Let the partially glued graph $\Lambda_5$ be as above and suppose that $\epsilon$ is sufficiently small.  There is a constant $C_5\in(0,1)$ with the following property.  Let $\mu_5$ be the distribution of bubble types in $\Omega_5$ and let $\nu$ be any integral distribution on bubble types such that, for each type $B$,
\[
\mu_5(B)(1-C_5)<\nu(B)\leq\mu_5(B)
\]
Then we may apply trades as above to 3 copies of the partially glued graph $\Lambda_5$ to produce a new partially glued graph $\Lambda_6$ such that:
\begin{enumerate}
\item the induced distribution $\mu_6$ on bubbles and bicrowns satisfies $\tilde{\mu}_6=3\nu$;
\item the mass of the remainder $\Delta_6$ is bounded above by
\[
\frac{m(\Delta_6)}{m(\Lambda_0)}\leq O(1/N);
\]
\item the slack satisfies $\theta_6\geq\theta_5/2$.
\end{enumerate}
\end{lemma}
\begin{proof}
Let $\mu_5$ be the distribution on football bubbles and bicrowns derived from $\Lambda_5$.  From the upper bound on the total mass of the remainder, it follows that $1\leq\tilde{\mu}_5/\mu_5\leq 1+O(1/N)$. 

We divide the inner half of the outer reservoir $\Omega^o_5$ into strips of five contiguous football bubbles, and call a football bubble \emph{fifth} if it lies in the center of such a quintuple.  Let $\bar{\mu}_5$ denote the distribution of fifth football bubbles in the inner half of the outer reservoir $\Omega^o_5$.
 
Recall that, in the outer reservoir, the bubbles are distributed within $O(c^N)\epsilon$ of the uniform distribution.  Therefore, for any football bubble type $B$, $|\bar{\mu}_5(B)-\rho|<O(c^N)\epsilon$ (where $\rho$ is the uniform distribution, scaled appropriately).  In particular, taking $N$ sufficiently large and $\epsilon$ sufficiently small, we have
\[
\bar{\mu}_5(B)> C_5\mu_5(B)
\]
for some constant $C_5$.  Hence the hypothesis of the lemma implies that $\bar{\mu}_5(B)> \mu_5(B)-\nu(B)$ for every type $B$ of football bubble. If $N$ is sufficiently large then it follows further that $\bar{\mu}_5(B)\geq \tilde{\mu}_5(B)-\nu(B)$ for every type $B$.

Consider each bubble in the center of a quintuple in the outer reservoir of type $B$.  We color the bubble black with probability $(\tilde{\mu}_5(B)-\nu(B))/\bar{\mu}_5(B)$ and white otherwise. 

We now construct $\Lambda_6$ from three copies of $\Lambda_5$, by performing a trade at each bubble colored black.  Taking three copies of $\Lambda_5$ triples the number of each bubble type. The lemma is phrased so that replacing three bubbles of a given type by a bicrown of corresponding covering type is neutral.  The only remaining effect of a trade is then to remove exactly three bubbles of the central type.  This proves the lemma.
\end{proof}

In the sequel, we will apply this lemma with a particular distribution $\nu$, described in Lemma \ref{lem: Cubical distribution} below.

\subsection{Cube and prism moves}\label{subsection:cube_move}

We have two more gluing steps: a small mass of bicrowns and bizenes must be glued up with footballs
(drawn from an almost equidistributed collection of much larger mass), then the distribution of
the footballs can be corrected by trades so that they are perfectly evenly distributed, and finally
an evenly distributed collection of footballs (i.e.\/ a collection with {\em exactly} the same number
of footballs with each possible label) must be entirely glued up.  We next describe three moves which will enable us to glue up bubbles, bizenes and bicrowns. 

\subsubsection{The cube move}

The idea is very simple: four footballs with appropriate edge labels can be draped over
the 1-skeleton of a cube in a manner invariant under the action of the Klein 4-group, and then
glued up according to how they match along the edges of the cube. This is indicated in
Figure~\ref{footballs_on_cube}.

\begin{figure}[htpb]
\labellist
\small\hair 2pt
%\pinlabel $\text{Type A}$ at -200 0
\endlabellist
\centering
\includegraphics[scale=0.7]{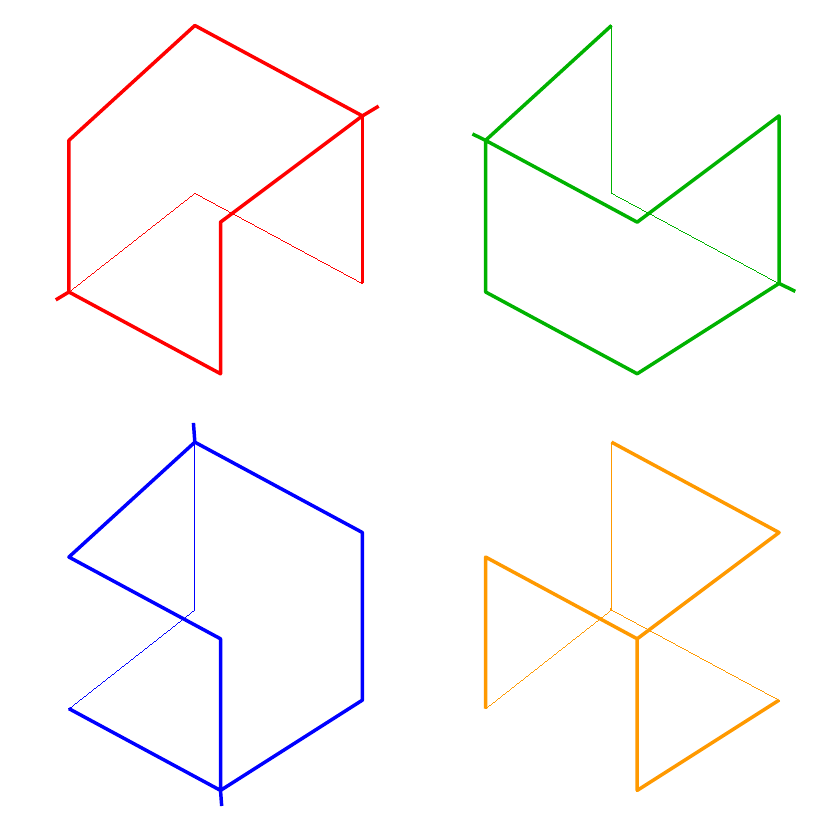}
\caption{Four footballs draped over the 1-skeleton of a cube}\label{footballs_on_cube}
\end{figure}

We next give an algebraic description of the cube move in terms of covering spaces of graphs, which enables us to give a precise description of the various coverings of the cube moves that we will also need.

Consider two theta-graphs, $\Theta$ and $H$.  The three edges of $\Theta$ are denoted by $\alpha,\beta,\gamma$ (oriented so they all pointing in the same direction), and likewise the three edges of $H$ are denoted by $a,b,c$. We consider the immersion $\Theta\to H$ which maps the edges of $\Theta$ to concatenations of edges of $H$ as follows:
\[
\alpha\mapsto a\bar{b}c~, ~\beta\mapsto b\bar{c}a~,~\gamma\mapsto c\bar{a}b
\]
(as usual, $\bar{a}$ denotes $a$ with the opposite orientation etc).  To enable us to reason group-theoretically, we set $\delta=\alpha\bar{\gamma}$ and $\epsilon=\beta\bar{\gamma}$, and similarly $d=a\bar{c}$ and $e=b\bar{c}$. Fixing base points at the initial vertices of all the edges, $\pi_1\Theta$ is the free group on $\delta,\epsilon$ and $\pi_1H$ is the free group on $d,e$.  We immediately see that the immersion $\Theta\to H$ induces the identifications
\[
\delta=\alpha\bar{\gamma}=a\bar{b}c\bar{b}a\bar{c}=(a\bar{c})(\overline{b\bar{c}})^2(a\bar{c})=de^2d
\]
and
\[
\epsilon=\beta\bar{\gamma}=b\bar{c}a\bar{b}a\bar{c}=(b\bar{c})(a\bar{c})\overline{(b\bar{c})}(a\bar{c})=ede^{-1}d
\]

The graph $\Theta$ should be thought of as a bubble, and the graph $H$ as a pattern for gluing it up.  In what follows, we will describe various covering spaces $H_\bullet\to H$. The fibre product $\Theta_\bullet$ of the maps $H_\bullet\to H$ and $\Theta\to H$, together with the induced map $\Theta_\bullet\to H_\bullet$, will then describe various gluing patterns for (covering spaces of) unions of bubbles.

We first start with the cube move itself.  Consider the natural quotient map $q_4:\pi_1H\to H_1(H;\Z/2\Z)\cong V$, where $V$ is the Klein 4-group.  The corresponding covering space $H_4\to H$ is a cube with quotient graph $H$. Note that the deck group $V$ acts freely on the cube $H_4$, freely permuting the diagonals.  Since $q_4(\pi_1\Theta)=1$, the fibre product $\Theta_4$ is a disjoint union of four copies of $\Theta$, each spanning a diagonal in the cube $H_4$ and freely permuted by $V$.  In particular, the map $\Theta_4\to H_4$ precisely defines the cube move.
%\footnote{If I've done the calculation right then in fact $V\cong \langle d,e\mid de^2d, ede^{-1}d \rangle$, and it follows that the cube move is the only regular cover of $H$ that provides a gluing move for bubbles. -hw}

\subsubsection{Gluing bizenes}

We next describe a gluing move for bizenes.  Consider the quotient map $q_8$ from $\pi_1H$ to the dihedral group $D_8=\langle\sigma,\tau\mid \sigma^4=\tau^2=\tau\sigma\tau\sigma=1\rangle$ defined by $d\mapsto\sigma$ and $e\mapsto\tau$, and let $H_8$ be the covering space of $H$ corresponding to the kernel of $q_8$.  Since $q_8$ factors through $q_4$, $H_8$ is a degree-two covering space of the cube $H_4$ 
(the graph $H_8$ is in fact the 1-skeleton of an octagonal prism).  We next calculate the restriction of $q_8$ to $\pi_1\Theta$:
\[
q_8(\delta)=q_8(de^2d)=\sigma\tau^2\sigma=\sigma^2
\]
while
\[
q_8(\epsilon)=q_8(ede^{-1}d)=\tau\sigma\tau\sigma=1
\]
The covering space of $\Theta$ corresponding to the restriction of $q_8$ is therefore a bizene.  It follows that the fibre product map
\[
\Theta_8\to H_8
\]
describes a double cover of the cube move, which glues four bizenes along an octagonal prism.  We will call this an \emph{8-prism move}.

\subsubsection{Gluing bicrowns}

Finally, we describe a gluing move for bicrowns.  Consider the quotient map $q_{12}:\pi_1H\to D_{12}=\langle\sigma,\tau\mid \sigma^6=\tau^2=\tau\sigma\tau\sigma=1\rangle$ defined by $d\mapsto\tau$ and $e\mapsto\sigma$.  Then, as before, since $q_{12}$ factors through $q_4$, the kernel of $q_{12}$ corresponds to a regular covering space $H_{12}\to H_4$ of degree three 
(in fact, $H_{12}$ is isomorphic to the 1-skeleton of a dodecagonal prism).  Again, we calculate the restriction of $q_{12}$ to $\pi_1\Theta$, and find that
\[
q_{12}(\delta)^{-1}=q_{12}(\epsilon)=\sigma^2
\]
(an element of order 3).  The covering space of $\Theta$ corresponding to the restriction of $q_{12}$ is therefore a bicrown.  In particular, the fibre product map
\[
\Theta_{12}\to H_{12}
\]
describes a triple cover of the cube move, which glues four bicrowns along a dodecagonal prism.  We will call this a \emph{12-prism move}.

\iffalse
Actually, although it is true that collection of footballs exactly equidistributed with respect to
type can be entirely glued up with cube moves (after multiplying through to clear denominators),
this is not straightforward to see. The reason is that in a uniform set of cube types (i.e.\/
all possible assignments of compatible labels of length $10\lambda$ to edges of a cube), the
set of football types that appear is not quite equidistributed --- it differs multiplicatively
from the uniform distribution by approximately $(2k-1)^{-10\lambda/3}$. So it is easiest just
to arrange in advance for the distribution of football types to match the distribution appearing
in the uniform distribution of cubes, by using the method of \S~\ref{subsection:distribution}.
This gets to within $\epsilon$ of the desired distribution; then we apply tear moves to the
remainder, and glue up the resulting bicrowns and bizenes with a small fraction of the mass of
footballs. Finally we adjust the mass of footballs with trades to correct the remaining error, 
and glue everything up with cube moves.
\fi

\subsection{Creating bizenes and bicrowns}

To complete the proof of the Thin Spine Theorem, we need to glue up the remaining small mass (of order $1/N$) of bizenes and bicrowns using prism moves, before gluing up the remaining football bubbles using cube moves.  We shall see that trades provide us with enough flexibility to do this, as long as $N$ is large enough.  However, since prism moves require that bizenes are glued up with bizenes and our bicrowns are glued up with bicrowns, and yet the reservoir consists only of football bubbles, we will need a move that turns football bubbles into bizenes and bicrowns of given types.

\subsubsection{Bicrown assembly}

Since the bicrowns that we need to assemble are all of covering type, it is straightforward to
construct them from football bubbles. Given a bicrown of covering type $\widetilde{B}$, covering a football bubble of type $B$, there is a move which takes as input three identical triples of bubbles each of type $B$, and transforms them into three bicrowns of type $\widetilde{B}$, without any other changes to the unglued subgraph.  
%\footnote{This is probably most easily explained with a picture. -hw}PICTURE

The following lemma is an immediate consequence of the fact that a 12-prism triply covers a cube.

\begin{lemma}\label{lem: Gluing bicrowns}
For every bicrown of covering type there exist three more bicrowns of covering type such that the four bicrowns together can be glued with a 12-prism move.
\end{lemma}

\subsubsection{Bizene assembly}

Bizene assembly is more subtle, because the bizenes that we need are more general than simply of covering type.  We assemble bizenes using the following move.

Consider an adjacent pair of bubbles of type $(a,b,c)$ and $(a',b',c')$, separated by a sticky strip of type $x$.  Suppose also that the second bubble (of type $(a',b',c')$) is followed by a further sticky strip also of type $x$.  Consider also a second pair of bubbles, of the same type but with the two bubbles swapped. From these two pairs we may construct two bizenes of the same type. The pairs of edges with the same start and end points are labeled $(a,b)$ and $(a',b')$, while the edges joining one pair to the other are labeled $c$ and $c'$.

The bizenes that we may assemble in this way satisfy some constraints, arising from the fact that both the ends of the following triples must all be compatible with the start of $x$: $(a,b,c)$, $(a',b',c')$, $(a,b,c')$, $(a',b',c)$.  In the rank-two case, this creates especially strong constraints, which (up to relabeling) can be simply stated as requiring that the ends of $a$, $b$ and $c$ should be equal to the ends of $a'$, $b'$ and $c'$ respectively.  We shall call such a bizene \emph{constructible}.

Just as the bizenes that we can construct are constrained, so the bizenes that we need to glue up from the remainder are also of a special form.  Indeed, in the proof of Lemma \ref{lem: Tearing} we were free to choose the interiors of the bubbles $B_1$ and $B_2$ in any way.

\begin{figure}[ht]
\begin{center}
 \centering \def\svgwidth{250pt}
 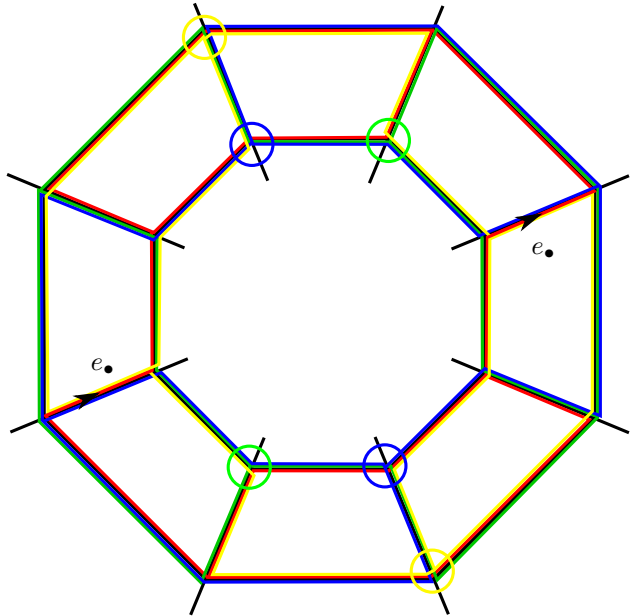
 \caption{Four bizenes draped over an octagonal prism.  A remainder bizene is colored red; the only constraint in its construction arises from the two identical arcs labeled $\alpha$.  The yellow, green and blue bizenes are constructible.  They are constrained only by the requirements that the circled vertices should be identically labeled.}\label{fig: Gluing bizenes}
\end{center}
\end{figure}

\begin{lemma}\label{lem: Gluing remainder bizenes}
There exist choices of the bubble types $B_1$ and $B_2$ in the proof of Lemma \ref{lem: Tearing} such that the resulting remainder bizenes can be glued with a prism move to three constructible bizenes.
\end{lemma}
\begin{proof}
Such a choice is illustrated in Figure \ref{fig: Gluing bizenes}.  Note that the constrained vertices of the yellow, green and blue constructible bizenes are disjoint from each other and from the determined arcs $e_\bullet$ of the red remainder bizene.  Therefore, we can start by labeling the two arcs $e_\bullet$ and the constrained vertices, and then label the rest of the 8-prism in any way we want. Doing this for each $e_\bullet$ determines the bubble types $B_1$ and $B_2$.
\end{proof}

\subsection{Gluing up the remainder}

In this section we use the moves described above to completely glue up the remainder and the reservoir.  But first we will address two important details which we have hitherto left undefined: the distributions $\mu$ (of Lemma \ref{lem: Adjusting the distribution}) and $\nu$ (of Lemma \ref{lemma:trades}).

We first describe `cubical distributions'.  Consider any distribution $\kappa$ on the set of types of cubes with side length $\lambda$.  The cube move associates to each type of cube a collection of four types of football bubbles.  The push forward of any distribution $\kappa$ to the set of types of football bubbles is called a \emph{cubical distribution}.  In particular, if $\kappa$ is the uniform distribution on the set of types of cubes then we call the push forward the \emph{uniform cubical distribution}, or just the \emph{cube distribution} for brevity.

In Subsection \ref{subsection:distribution} above we may take $\mu$ to be the cube distribution, so that the set of football bubbles in the reservoir $\Omega_5$ is within $O(1/N)$ of $\mu'$, a distribution proportional to the cube distribution.

We next address the distribution $\nu$ from Lemma \ref{lemma:trades}.  It consists of two parts: any cubical distribution $\kappa$, and a \emph{bizene correction distribution} $\beta$.  That is, $\nu=\kappa+\beta$.  So we need to describe the bizene correction distribution.

The remainder $\Delta_5$ consists of (remainder) bizenes and bicrowns.  By Lemma \ref{lem: Gluing remainder bizenes}, to each remainder bizene we associate (some choice of) three constructible bizenes.  Each constructible bizene can in turn be constructed from a pair of types of football bubble. Thus, to each remainder bizene we associate six football bubbles.  Summing over all bizenes in the remainder $\Delta_5$ defines the distribution $\beta$. 

In order to apply Lemma \ref{lemma:trades}, we need to check that there is a cubical distribution $\kappa$ such that $\nu=\kappa+\beta$ satisfies the hypotheses of the lemma.

\begin{lemma}\label{lem: Cubical distribution}
If $N$ is sufficiently large then there exists a cubical distribution $\kappa$ such that the integral distribution $\nu=\kappa+\beta$ satisfies
\[
1-C_5\leq\nu/\mu_5\leq 1
\]
(where $C_5$ is the constant from Lemma \ref{lemma:trades}).  Furthermore, as long as $m(L)$ is sufficiently large, we may take $\kappa$ to be integral.
\end{lemma}
\begin{proof}
Since the mass of the remainder is $O(1/N)m(L)$ and $\mu_5$ is bounded below, it follows that $\beta(B)\leq O(1/N)\mu_5(B)$ for each type $B$ of football bubble, so it suffices to show that there is an integral cubical distribution $\kappa$ satisfying
\[
1-C_5 \leq \kappa/\mu_5\leq 1-O(1/N)
\]
By the construction of $\Lambda_5$, there is a cubical distribution $\mu'$ such that $|1-\mu'/\mu_5|<O(1/N)$. Choose a rational $\eta\in(1-C_5,1)$.  As long as $N$ is sufficiently large we will also have that $1-C_5+O(1/N)<\eta<1-O(1/N)$, and it follows that $\kappa=\eta\mu'$ satisfies the required condition.  Furthermore, if $m(L)$ is sufficiently large then $\eta$ can be chosen so that $\kappa$ is integral.
%\footnote{More needs to be said here. -hw}
\end{proof}

We can now glue up all the bizenes, using bizene assembly and the 8-prism move.

\begin{lemma}
Let $\Lambda_6$ be as in Lemma \ref{lemma:trades}, using the distribution $\nu$ from Lemma \ref{lem: Cubical distribution}.  Then we may apply 8-prism moves to 2 copies of $\Lambda_6$ to produce a partially glued graph $\Lambda_7$ such that:
\begin{enumerate}
\item every component of the remainder $\Delta_7$ is a bicrown of covering type;
\item the total mass of the remainder $\Delta_7$ satisfies $m(\Delta_7)/m(\Lambda_7)\leq O(1/N)$;
\item if $\mu_7$ is the distribution of bubbles and bicrowns in $\Lambda_7$ then $\tilde{\mu}_7$ is cubical;
\item the slack $\theta_7$ satisfies $\theta_7\geq\theta_6/2$.
\end{enumerate}
\end{lemma}
\begin{proof}
Let $\tilde{\beta}$ be the distribution of constructible bizenes required to glue up the remainder bizenes in $\Lambda_6$.  Take two copies of $\Lambda_6$.  Using the bizene assembly move, we construct exactly $2\tilde{\beta}(B)$ new bizenes of each type $B$ from the inner half of the outer reservoir.  From the definition of $\tilde{B}$ we may now glue up all the bizenes using the 8-prism move.  By the construction of $\Lambda_6$, it follows that $\tilde{\mu}_7=2\tilde{\mu}_6-2\beta=2\kappa$ and so is cubical.

Since the total mass of bicrowns was $O(1/N)$ in $\Lambda_6$, the same is true in $\Lambda_7$.
\end{proof}

The next lemma completes the proof of the Thin Spine Theorem, except for a small adjustment needed to correct co-orientation, in the case when $k>2$.

\begin{lemma}
From 3 copies of $\Lambda_7$ as above, we can construct a graph $\Lambda_8$ in which the unglued subgraph $\Upsilon_8$ is empty.
\end{lemma}
\begin{proof}
For each bicrown $\widetilde{B}_0$ (of covering type), there exist bicrowns of covering type $\widetilde{B}_i$ (where $i=1,2,3$) such that the $\widetilde{B}_i$ for $i=0,1,2,3$ can be glued up using a 12-prism move. Three copies of each of these $\widetilde{B}_i$ can in turn be constructed from three consecutive copies of bubbles $B_i$, using the bicrown assembly move from Lemma \ref{lem: Gluing bicrowns}.   Let $\alpha$ be the distribution on bubble types that, for each bicrown of type $\widetilde{B}_0$ in the remainder, counts three bubbles of each type $B_i$.  Note that, because all the bicrowns are of covering type and the 12-prism move covers the cube move, the distribution $\alpha$ is cubical.

The partially glued graph $\Lambda_8$ is constructed from three copies of $\Lambda_7$.  Since the mass of the remainder $\Delta_7$ is bounded above by
\[
\frac{m(\Delta_7)}{m(\Lambda_7)}\leq O(1/N)
\]
whereas the mass of the outer reservoir $\Omega^o_7$ is bounded below by
\[
\frac{m(\Omega^o_7)}{m(\Omega_7)}\geq \theta_7\geq O(1)~,
\]
for $N$ sufficiently large we may use bicrown assembly to construct three times the number of bicrowns needed to glue up the remainder, using football bubbles from the outer reservoir.  We can then use 12-prism moves to glue up all the bicrowns.

The distribution of the remaining football bubbles  is still cubical, and so they can also be glued up with cube moves.
\end{proof}

\subsection{The co-orientation condition}\label{subsection:coorientation}

The result of all this gluing is to produce $f:L \to Z$ which is degree 3, and whose local model
at every vertex of $Z$ is good. What remains is to check that the construction can be done while
satisfying the co-orientation condition. The obstruction to this condition can be thought of as
an element of $H^1(L;\Z/2\Z)$. Since $L$ has a bounded number of components,
it should not be surprising that we can adjust the gluing by local moves to ensure the 
vanishing of the co-orientation obstruction. In fact, it is easier to arrange this after
taking 2 disjoint copies of $L$, and possibly performing a finite number of moves, which we
now describe.

After the first gluing step, we trivialize the $I$-bundles (in an arbitrary way) 
along the preimage of each of the short segments. This trivialization determines a 
{\em relative} co-orientation cocycle on each football bubble or component of the remainder;
we refer to this relative cocycle as a {\em framing}. The set of framings of each
component $\gamma$ is a torsor for $H^1(\gamma;\Z/2\Z)$; the four possible framings of a 
football bubble are depicted in Figure~\ref{framed_bubbles}.

\begin{figure}[htpb]
\labellist
\small\hair 2pt
%\pinlabel $\text{Type A}$ at -200 0
\endlabellist
\centering
\includegraphics[scale=1]{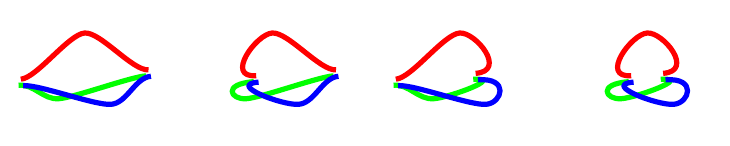}
\caption{Four framings on a football bubble}\label{framed_bubbles}
\end{figure}

Subsequent moves all make sense for framed bubbles, bizenes, bicrowns and so on. Each gluing
move can be done locally in a way which embeds in $\R^3$ (embeddings are illustrated in the figures
throughout the last few sections); such an embedding determines a
{\em move framing} on each of the pieces. The difference between a given framing and the
move framing determines a class in $H^1(\gamma;\Z/2\Z)$ for each piece, and the sum over
all such pieces is the (global) co-orientation cocycle in $H^1(L;\Z/2\Z)$.

There is a very simple procedure to adjust this global co-orientation cocycle, which we
now describe. Suppose we have a pair of footballs $\gamma$ and $\gamma'$ with the same
3 labels, but with framings which differ by a single reflection at a vertex (i.e.\/ they
are of the first two types depicted in Figure~\ref{framed_bubbles}). Swapping
$\gamma$ and $\gamma'$ in two cube moves that they participate in adjusts the cocycle six
times, once for each of the six edges in the two bubbles; we call this a {\em swap} move. 
If three of these edges are in
a component $L_i$, and three in $L_j$, then the global change to the cocycle is to add
a fundamental class of $H^1(L_i \cup L_j;\Z/2\Z)$. If we performed our original gluing
randomly, every component $L_i$ should contain many pairs of footballs with framings
which differ in this way. So if we take two disjoint copies of $L \to Z$, we can trivialize
the co-orientation cocycle by finitely many such swaps. 
This duplication multiplies the total number of components of $L$ by a factor of 2.

This completes the proof of the Thin Spine Theorem~\ref{theorem:thin_spine_theorem}, at least
when $k\ge 3$.

\subsection{Rank 2}\label{rank_2_subsection}

The move described in \S~\ref{subsection:supercompatible} to deal with an excess of
$O(\epsilon)$ long strips requires rank $k\ge 3$, so that long
strips can be grouped into super-compatible 4-tuples if necessary. In this section we briefly
explain how to finesse this point in the case $k=2$.

Fix some constant $C$ with $1\ll C \ll N$; $C$ will need to satisfy some divisibility
properties in what follows, but we leave this implicit. Define a
{\em pocket} to be three equal segments of length $3(C+2)\lambda$ which can be glued compatibly.
Recall at the very first step of our construction that we glued compatible long strips in 
triples by bunching sticky segments to form bubbles. We modify this construction slightly
by also allowing ourselves to create some small mass of bunched pockets. That is, we bunch triples
of long strips of the form
$$a_0 x_1 a_2 x_3\cdots a_{N'}, \quad b_0 x_1 b_2 x_3 \cdots b_{N'}, \quad c_0 x_1 c_2 x_3 \cdots c_{N'}$$
if for even $i$ the letters adjacent to each $x_{i+1}$ or $x_{i-1}$ disagree, where each
$a_i$, $b_i$ or $c_i$ has length $3\lambda$, and where each $x_j$ {\em either} has length 
$3\lambda$, or has length $3(C+2)\lambda$. We insist that the proportion of $x_j$ of length
$3(C+2)\lambda$ is very small, so that most of the bunched triples are of length $3\lambda$,
but that some small mass of bunched pockets has also been created. Note that $N'$ will depend
on the number of ``long'' $x_j$, but in any case $N'$ will be quite close to $N$.

We arrange by pseudorandomness
that the mass of bunched pockets of every possible type is $O(\epsilon)$, but with a constant
such that this mass is definitely larger than the number of long strips that will remain
unbunched after the first step.

Now consider a long strip $\sigma$ left unbunched after the first step. 
We partition this strip in a different way as
$$\sigma:=e_0 z_1 e_2 z_3 \cdots e_M$$
where each $e_j$ has length $3\lambda$, and each $z_j$ has length $C\lambda$.

We take 9 copies of $\sigma$. Fix an index $j$. 
Suppose we have two sets of three bunched triples of pockets (hence 18 pockets in
all) of the form
$$a_i\alpha_i x_i \beta y_i \gamma_i r_i, \; b_i\alpha_i x_i \beta y_i \gamma_i s_i, \;
c_i\alpha_i x_i \beta y_i \gamma_i t_i \text{ bunched along the pocket } \alpha_ix_i\beta y_i\gamma_i$$
$$a_i'\alpha_i' x_i' \beta' y_i' \gamma_i' r_i', \; b_i'\alpha_i' x_i' \beta' y_i' \gamma_i' s_i', \;
c_i'\alpha_i' x_i' \beta' y_i' \gamma_i' t_i'\text{ bunched along the pocket } \alpha_i'x_i'\beta' y_i'\gamma_i'$$
for each of $i=1,2,3$, and satisfying
\begin{enumerate}
\item{each of $a_i,b_i,c_i,r_i,s_i,t_i,x_i,y_i$ and their primed versions have length $3\lambda$;}
\item{each of $\alpha_i,\beta,\gamma_i$ and their primed versions have length $C\lambda$;}
\item{$a_1,a_2,a_3$ all end with the same letter and $r_1,r_2,r_3$ all start with the
same letter, and similarly for $b_i,c_i,s_i,t_i$ and the primed versions;}
\item{$a_i,b_i,c_i$ end with different letters and $r_i,s_i,t_i$ start with different letters for
each fixed $i$, and similarly for the primed versions;}
\item{$x_1,x_2,x_3$ start with different letters and end with the same letter and similarly
for the primed versions;}
\item{$y_1,y_2,y_3$ end with different letters and start with the same letter and similarly for
the primed versions;}
\item{$\alpha_1,\alpha_2,\alpha_3$ can be partitioned into an odd number of segments 
of length $3\lambda$ which can be compatibly bunched creating a strip of alternate short segments 
and bubbles, and similarly for the $\gamma_i$ and the primed versions;}
\item{the common last letter of the $x_i$ is different from the common last letter of the $x_i'$
and from the last letter of $e_j$, and
the common first letter of the $y_i$ is different from the common first letter of the $y_i'$ and
from the first letter of $e_{j+2}$; and}
\item{$\beta$, $\beta'$ and $z_j$ can be partitioned into an odd number of segments of length 
$3\lambda$ which can be compatibly bunched creating a strip of alternate short segments
and bubbles.}
\end{enumerate}

Under these hypotheses, we can pull apart the six bunched pockets, bunch the $\alpha_i$ 
in short strips (and similarly bunch the $\alpha_i'$), bunch the $\gamma_i$ in short strips
(and similarly bunch the $\gamma_i'$), and finally bunch the three sets of
$\beta$, $\beta'$ and $z_j$ in short strips.
Explicitly, we are creating bunched segments of length $C\lambda$ of the following kinds:
$$a_1\alpha_1 x_1,\; b_2 \alpha_2 x_2, \; c_3 \alpha_3 x_3;\quad b_1\alpha_1 x_1,\; c_2 \alpha_2 x_2, \; a_3 \alpha_3 x_3; \quad c_1\alpha_1 x_1,\; a_2 \alpha_2 x_2, \; b_3 \alpha_3 x_3;$$
$$a_1'\alpha_1' x_1',\; b_2' \alpha_2' x_2', \; c_3' \alpha_3' x_3';\quad b_1'\alpha_1' x_1',\; c_2' \alpha_2' x_2', \; a_3' \alpha_3' x_3'; \quad c_1'\alpha_1' x_1',\; a_2' \alpha_2' x_2', \; b_3' \alpha_3' x_3';$$
$$y_1\gamma_1 r_1,\; y_2\gamma_2 s_2, \; y_3\gamma_3 t_3; \quad y_1\gamma_1 s_1,\; y_2\gamma_2 t_2, \; y_3\gamma_3 r_3; \quad y_1\gamma_1 t_1,\; y_2\gamma_2 r_2, \; y_3\gamma_3 s_3;$$
$$y_1'\gamma_1' r_1',\; y_2'\gamma_2' s_2', \; y_3'\gamma_3' t_3'; \quad y_1'\gamma_1' s_1',\; y_2'\gamma_2' t_2', \; y_3'\gamma_3' r_3'; \quad y_1'\gamma_1' t_1',\; y_2'\gamma_2' r_2', \; y_3'\gamma_3' s_3';$$
and finally three copies of
$$x_i\beta y_i, \; x_i'\beta'y_i', \; e_j z_{j+1} e_{j+2}$$
for each of $i=1,2,3$.

If we do this for each $z_j$ in turn, then the
net effect is to pair up all the extra long strips, at the cost of creating a new
remainder of mass $O(\epsilon)$, and using up mass $O(\epsilon)$ of the bunched pockets.

At the end of this step every vertex of the new remainder created is adjacent to
a strip of $C/3$ consecutive short bubbles; because of this, there is ample slack to apply
tear moves to the new remainder as in \S~\ref{subsection:tear}. Note that this move requires
us to take nine copies of each excess long strip; thus we might have to take a total of
5,832 copies of $L$ instead of 648 for $k\ge 3$. The rest of the argument goes through as above.
This completes the proof in the case $k=2$ and thus in general.

\section{Bead decomposition}\label{bead_decomposition_section}

The next step of the argument is modeled very closely on \S~5 from \cite{Calegari_Walker}. For the
sake of completeness we explain the argument in detail. Throughout this section we fix a free group
$F_k$ with $k\ge2$ generators and we let $r$ be a random cyclically reduced word of length $n$, and
consider the one-relator group $G:=\langle F\; | \; r \rangle$ with presentation complex $K$. 
Using the Thin Spine Theorem, we will construct (with overwhelming probability)
a spine $f:L \to Z$ over $K$ for which every edge of $Z$ has length at least $\lambda$, 
for some big $\lambda$. The main result of this section is that if this construction is done
carefully, the immersion $\overline{M}(Z) \to K$ will be $\pi_1$-injective, again with
overwhelming probability.

\subsection{Construction of the beaded spine}

A random 1-relator group satisfies the small cancellation property $C'(\mu)$ for every
positive $\mu$, with overwhelming probability. So to show that $\overline{M}(Z) \to K$ is
$\pi_1$-injective, it suffices to show that for any sufficiently long immersed segment 
$\gamma \to Z$ whose image in $X$ under $g:Z \to X$ lifts to $r$, it already lifts to $L$.
Informally, the only long immersed segments in $Z$ which are ``pieces'' of $r$ are those
that are in the image of segments of $L$ under $f:L \to Z$.

Let $Z$ be a 4-valent graph with total edge length $|Z|=O(n)$, in which every edge has length 
$\ge \lambda$. For any $\ell$, there are at most $|Z|\cdot 3^{\ell/\lambda}$ immersed paths in $Z$ of 
length $\ell$. Thus if $\ell$ is of order $n^\alpha$, and $\lambda \gg 1$, we would not expect to find any
paths of length $\ell$ in common with an {\em independent} random relator $r$ of length $n$,
for any fixed $\alpha > 0$, with probability $1-O(e^{-n^c})$ for some $c$ depending on $\alpha$.

There is a nice way to express this in terms of {\em density}; or {\em degrees of freedom},
which is summarized in the following {\em intersection formula} of Gromov; see
\cite{Gromov}, \S~9.A for details:

\begin{proposition}[Gromov's intersection formula]\label{proposition:intersection}
Let $C$ be a finite set. For a subset $A$ of $C$ define the (multiplicative) 
{\em density} of $A$, denoted $\density(A)$, to be
$\density(A):=\log{|A|}/\log{|C|}$. 
If $A_1$ is any subset of $C$, and $A_2$ is a random subset of $C$
of fixed cardinality, chosen independently of $A_1$, then with probability
$1-O(|C|^{-c})$ for some $c>0$, there is an equality
$$\density(A_1 \cap A_2) = \density(A_1) + \density(A_2) - 1$$
with the convention that $\density < 0$ means a set is empty.
\end{proposition}
Note that Gromov does not actually estimate the probability that his formula holds, but this
is an elementary consequence of Chernoff's inequality. For a proof of an analogous estimate,
which explicitly covers the cases of interest that we need, 
see \cite{Calegari_Walker_RR}, \S~2.4.

In our situation, taking $\ell = n^\alpha$, we can take $C$ to be the set of all reduced
words in $F_k$ of length $\ell$, which has cardinality approximately $(2k-1)^{n^\alpha}$.
If $\lambda \gg 1$, then the set of immersed paths in $Z$ of length $\ell$ has density as
close to $0$ as desired; similarly, the set of subwords of a random word $r$ of length $n$
has density as close to $0$ as desired. Thus {\em if these subwords were independent},
Gromov's formula would show that they were disjoint, with probability $1-O(e^{-n^c})$.

Of course, the thin spines $Z$ guaranteed by the Thin Spine Theorem are hardly independent of $r$.
Indeed, {\em every} subpath of $r$ appears in $L$, and therefore in $Z$! Thus, we must work
harder to show that these subpaths (those that are already in $L$) amount to all the intersection.
The idea of the bead decomposition is to subdivide 
$r$ into many subsets $b_i$ of length $n^{1-\delta}$ (for some fixed $\delta$), 
to build a thin spine $Z_i$ ``bounding'' the subset $b_i$, and then to argue that 
no immersed path in $Z_i$ of length $n^\alpha$ can be a piece in any $b_j$ with $i \ne j$.

Fix some small positive constant $\delta$, and write $r$ as a product
$$r = r_1 s_1 r_2 s_2 \cdots r_m s_m$$
where each $r_i$ has length $n^{1-\delta}$ and each $s_i$ has length approximately
$n^\delta$ (the exact values are not important, just the order of magnitude). Thus $m$ is
approximately equal to $n^{\delta}$; we further adjust the lengths of the $r_i$ and $s_i$
slightly so that $m$ is divisible by $3$. 

We say a reduced word $x$ has {\em small self-overlaps} if the length of the biggest proper
prefix of $x$ equal to a proper suffix is at most $|x|/3$. Almost every reduced word
of fixed big length has small overlaps.
Fix some positive constant $C < \delta/\log(2k-1)$, and for each index $i$ mod $m/3$ we look
for the first triple of subwords of the form
$a_1xa_2,b_1xb_2,c_1xc_2$ in $s_i,s_{i+m/3},s_{i+2m/3}$ such that
\begin{enumerate}
\item{the $a_i,b_i,c_i$ are single edges;}
\item{$a_1,b_1,c_1$ are distinct and $a_2,b_2,c_2$ are distinct;}
\item{$x$ has length $C\log{n}$ with $C$ as above; and}
\item{$x$ has small self-overlaps.}
\end{enumerate}
Actually, it is not important that $x$ has length exactly $C\log{n}$; it would be fine for
it to have length in the interval $[C\log{n}/2,C\log{n}]$, for example.
Any reduced word of length $C\log{n}$ with $C<\delta/\log(2k-1)$ will appear many times
in any random reduced word of length $n^\delta$, with probability $1-O(e^{-n^c})$ for some
$c$ depending on $\delta$. See e.g.\/ \cite{Calegari_Walker_RR}, \S.~2.3.
Then for each index $i$ mod $m/3$, the three copies of $x$ can be glued to produce
unusually long bunched triples $l_i$ that we call {\em lips}. The lips partition the remainder
of $r$ into subsets which we denote $b_i$, where the index $i$ is taken mod $m/3$, so that
each $b_i$ is the union of three segments of length approximately equal to $n^{1-\delta}$
consisting of $r_i,r_{i+m/3},r_{i+2m/3}$ together with the part of the adjacent $s_j$
outside the lips. We call the $b_i$ {\em beads}, and we call the partition of $r$ minus 
the lips into beads the {\em bead decomposition}.

Now we apply the Thin Spine Theorem to build a thin spine $f:L \to Z$ such that
\begin{enumerate}
\item{$L$ consists of 648 copies of $r$ (or 5,832 copies if $k=2$);}
\item{$Z$ is cyclically subdivided by the lips $l_i$ into connected subgraphs $Z_i$;}
\item{the 648 copies of $b_i$ in $L$ are precisely the part of $r$ mapping to the $Z_i$,
and the remainder of $L$ consists of segments mapping to the lips as above.}
\end{enumerate}
We call the result a {\em beaded spine}.

\begin{lemma}[No common path]\label{lemma:no_common_path_ij}
For any positive $\alpha$, we can construct a beaded spine with the property that 
no immersed path in $Z_i$ of length $n^\alpha$ can be a piece in any $b_j$ with $i\ne j$
mod $m/3$, with probability $1-O(e^{-n^c})$, where $c$ depends on $\alpha$.
\end{lemma}
\begin{proof}
The construction of a beaded spine is easy: with high probability, the labels on each $b_i$
are $(T,\epsilon)$-pseudorandom for any fixed $(T,\epsilon)$, and we can simply apply 
the construction in the Thin Spine Theorem to each $b_i$ individually to build $Z_i$, 
and correct the co-orientation once at the end by a local modification in $Z_1$ (say). 

By the nature of the bead decomposition, the $b_i$ are independent of each other. By
thinness, there are $O(n\cdot 3^{\ell/\lambda})$ immersed paths in $Z_i$ of length 
$\ell$. For any fixed positive $\alpha$, if we set
$\ell = n^\alpha$, and choose $\lambda$ big enough, then the density of this set of
paths (in the set of all reduced words of length $\ell$) is as close to $0$ as we like.
Similarly, the set of subwords in $b_j$ of length $\ell$ has density as close to $0$
as we like for big $n$. But now these subwords are {\em independent} of the immersed
paths in $Z_i$, so by the intersection formula (Proposition~\ref{proposition:intersection}), 
there are no such words in common, with probability $1-O(e^{-n^c})$.
\end{proof}

\subsection{Injectivity}

We now show why a beaded spine gives rise to a $\pi_1$-injective map of a 
3-manifold $\overline{M}(Z) \to K$. First we prove another lemma, which is really the
key geometric point, and will be used again in \S~\ref{endgame_section}:

\begin{lemma}[Common path lifts]\label{lemma:common_path_lifts}
For any positive $\beta$, with probability $1-O(e^{-n^c})$
we can construct a beaded spine $L \to Z \to X$
with the property that any immersed segment $\gamma \to Z$ with $|\gamma|=\beta n$
whose image in $X$ under $Z \to X$ lifts to $r$ or $r^{-1}$, already lifts to $L$.
\end{lemma}
In words, this lemma says that any path in $r$ of length $\beta n$ which immerses in $Z$
lifts to $L$, and therefore appears in the boundary of a disk of $\overline{M}(Z)$.
\begin{proof}
The proof follows very closely the proof of Lemma~5.2.3 from \cite{Calegari_Walker}.

First, fix some very small $\alpha$ with $\alpha/\log{2k-1} \ll C'$
where $C'\log{n}$ is the length of the lips in the beaded spine. This ensures that
a random word of length $n^\alpha$ is very unlikely to contain two copies of any
word of length $C'\log{n}$ with small self-overlaps; see e.g.\/ \cite{Calegari_Walker_RR}
Prop.~2.6 and Prop.~2.11 which says that the likelihood of this occurrence is 
$O(n^{-C})$ for some $C$.

Now, let $f_Z:\gamma \to Z$ be an immersed path of length $\beta n$ whose label is a subpath
of $r$ or $r^{-1}$; this means that there is another immersion 
$f_L:\gamma \to L$ such that the compositions $\gamma \to Z \to X$ and
$\gamma \to L \to X$ agree.

Using $f_L$, we decompose $\gamma$ into subpaths $\gamma_j$ which are the preimages of
the segments of $b_j$ under $f_L$. Apart from boundary terms, each of these $\gamma_j$ has length 
approximately $n^{1-\delta}$. By Lemma~\ref{lemma:no_common_path_ij}, no $\gamma_j$ 
can immerse in $Z_i$ with $i \ne j$ mod $m/3$ unless $|\gamma_j| < n^\alpha$, 
and in fact $f_Z$ must therefore take all of $\gamma_j$
into $Z_j$ except possibly for some peripheral subwords of length
at most $n^\alpha$. 

But this means that for all $j$, there is a subpath $\sigma \subset \gamma$ centered at the
common endpoint of $\gamma_j$ and $\gamma_{j+1}$, with $|\sigma| \le n^\alpha$, for which
one endpoint maps under $f_Z$ into $Z_j$ and the other into $Z_{j+1}$. 
This means that $f_Z$ must map $\sigma$ over the lip of $Z$ 
separating $Z_j$ from $Z_{j+1}$, and must contain a copy of the
word $x_j$ on that lip. But under the map $f_L$, the word $\sigma$ contains another
copy of $x_j$. By our hypothesis on $\alpha$, the probability that $\sigma$
contains {\em two} copies of $x_j$ is $O(n^{-C})$. If these copies are the same, then
the composition $f_L:\gamma \to L \to Z$ and $f_Z:\gamma \to Z$ agree on $\sigma$. But
since $Z \to X$ is an immersion, it follows that $f_L:\gamma \to L \to Z$ and $f_Z:\gamma \to Z$
must agree on {\em all} of $\gamma$; i.e.\/ that $f_Z$ lifts to $L$, as claimed. 

So the lemma is proved unless there are two distinct
copies of $x_j$ within distance $n^\alpha$ of each lip in the image of $\gamma$.
Since $\gamma$ has length $\beta n$, there are $n^\delta$ such lips; the probability of two
distinct copies for each lip is $O(n^{-C})$, and the probabilities for distinct lips are
independent, so the total probability is $O(e^{-n^c})$ and we are done.
\end{proof}

An immediate corollary is the existence of 3-manifold subgroups in random
1-relator groups:

\begin{proposition}\label{proposition:1_relator_injective}
Let $G=\langle F_k\;|\; r\rangle$ be a random 1-relator group where $|r|=n$. Then with
probability $1-O(e^{-n^c})$ we can produce a beaded spine $L \to Z \to X$ for which
the associated map $\overline{M}(Z) \to K$ is $\pi_1$-injective.
\end{proposition}
\begin{proof}
Suppose not, so that there is an immersed loop $\gamma:S^1 \to Z$ 
which is nontrivial in $\pi_1(\overline{M}(Z))$, but trivial in $K$.
There is a van Kampen diagram $\D$ with boundary $\gamma$.
If $D$ is a disk in this diagram with some segment in common with $\gamma$, and if
$\partial D \to Z$ lifts to $L$, then $\partial D$ bounds a disk in $\overline{M}(Z)$, and
we can push $\D$ across $D$ by a homotopy, producing a diagram with fewer disks.
A diagram which does not admit such a simplification is said to be {\em efficient}; 
without loss of generality therefore
we obtain an efficient diagram whose boundary is an immersed loop $\gamma:S^1\to Z$.

The group $G$ satisfies the small cancellation property $C'(\mu)$ for any positive $\mu$.
Thus by Greedlinger's Lemma, if we take $\mu$ small enough, some disk $D$ in the diagram
has a segment of its boundary of length at least $n/2$ in common with $\gamma$. Note that $\partial D$
is labeled $r$ or $r^{-1}$, and has length $n$. Since $1/2 > \beta$ as
in Lemma~\ref{lemma:common_path_lifts}, the boundary of this path actually lifts to $L$,
whence the diagram is not efficient after all. This contradiction proves the theorem.
\end{proof}

Note that each boundary component of $\overline{M}(Z)$ is of the form $\overline{S}(Y)$ 
for some 3-valent fatgraph $Y$ immersed in $Z$; thus the same argument implies 
that every component of $\partial \overline{M}(Z)$ is $\pi_1$-injective, 
and therefore $\overline{M}(Z)$ has incompressible boundary. Furthermore, since $K$ is
aspherical, so is $\overline{M}(Z)$, and therefore $\overline{M}(Z)$ is irreducible, and
$\pi_1(\overline{M}(Z))$ does not split as a free product.

To show that $\overline{M}(Z)$ is homotopic to a hyperbolic 3-manifold with totally geodesic
boundary, it suffices to show that it is acylindrical, by Thurston's hyperbolization theorem---see, for instance, \cite[Theorem 4.3]{Bonahon_Structures}.
We show this in \S~\ref{acylindricity_section}.

\section{Acylindricity}\label{acylindricity_section}

In this section we explain why the 3-manifolds we have produced in random 1-relator groups are
acylindrical.

\begin{theorem}[1-Relator Acylindrical Subgroup Theorem]\label{theorem:1_relator_acylindrical}
Let $G=\langle F_k \; | \; r\rangle$ be a random 1-relator group where $|r|=n$. Then with
probability $1-O(e^{-n^c})$ for the beaded spine $L \to Z \to X$ guaranteed 
by Proposition~\ref{proposition:1_relator_injective} 
the 3-manifold $\overline{M}(Z)$ is acylindrical. Thus, with overwhelming probability,
random 1-relator groups contain subgroups isomorphic to the fundamental group of a hyperbolic
3-manifold with totally geodesic boundary.
\end{theorem}
\begin{proof}
Each boundary component $\partial_i \subset \partial \overline{M}(Z)$ 
is of the form $\overline{S}(Y_i)$ for some trivalent fatgraph $Y_i \to Z$ immersed in $Z$.
Suppose $\overline{M}(Z)$ admits an essential annulus. Then there is a van Kampen diagram
on an annulus $\A$ with boundary $\gamma_1,\gamma_2$ where $\gamma_i$ immerses in $Y_i$
for fatgraphs $Y_i$ associated to boundary components $\partial_i$ as above, and
each $\gamma_i$ is essential in $\overline{S}(Y_i)$ (which in turn is essential in
$\overline{M}(Z)$).

Assume that $\A$ is efficient. Then  by Lemma~\ref{lemma:common_path_lifts},
for any positive $\beta$ we can insist that no segment of $\partial D \cap \gamma_i$ 
has length more than $\beta n$. For big $n$, with probability $1-O(e^{-n^c})$ we
know that $K$ is $C'(\mu)$ for any positive $\mu$; when $\mu$ is small, the annular version
of Greedlinger's Lemma (see \cite{Lyndon_Schupp} Ch.~V Thm.~5.4 and its proof) 
implies that if $\A$ contains a disk at all, then
some disk $D$ in the diagram has a segment on its boundary of length at least $n/3$ in
common with $\gamma_1$ or $\gamma_2$. Taking $\beta < 1/3$ we see that $\A$ can contain
no disks at all; i.e.\/ $\gamma_1$ and $\gamma_2$ have the same image $\gamma$ in $Z$.

Now we use the fact that $L \to Z$ is a good spine. At each vertex $v$ of $Z$, four local boundary
components of $\overline{M}(Z)$ meet; the fact that $\gamma$ lifts to paths $\gamma_1$
and $\gamma_2$ in two of these component forces $\gamma$ to run between two specific edges
incident to $v$, and this determines a {\em unique} lift of $\gamma$ to $L$ near $v$
compatible with the existence of the $\gamma_i$. Similarly, along each edge
$e$ of $Z$, three local boundary components of $\overline{M}(Z)$ meet; the components
containing $\gamma_1$ and $\gamma_2$ thus again determine a unique lift of $\gamma$ to $L$
along $e$. These local lifts at vertices and along edges are compatible, and determine a
{\em global} lift of $\gamma$ to $L$. But this means $\gamma$ is inessential in 
$\overline{M}(Z)$, contrary to hypothesis. Thus $\overline{M}(Z)$ is acylindrical after all.
\end{proof}

\section{3-Manifolds Everywhere}\label{endgame_section}

We now show that the acylindrical 3-manifold subgroups that we have constructed in
random 1-relator groups stay essential as we add $(2k-1)^{Dn}$ independent random relations
of length $n$, for any $D<1/2$. Our argument follows the proof of
Thm.~6.4.1 \cite{Calegari_Walker} exactly, and depends only on the following two facts:
\begin{enumerate}
\item{the beaded spine $Z$ has total length $O(n)$, has valence 4, and every segment
has length at least $\lambda$, where we may choose $\lambda$ as big as we like (depending on $D$); and}
\item{any immersed segment $\gamma \to Z$ of length $\beta n$ whose label is a subword
of $r$ or $r^{-1}$ lifts to $L$, where we may choose $\beta$ as small as we like (depending on $D$).}
\end{enumerate}

Beyond these facts, we use two theorems of \cite{Ollivier}, which give explicit
estimates for the linear constant in the isoperimetric function and 
for the constant of hyperbolicity for a random group at density $D<1/2$.

\subsection{Ollivier's estimates}

We use the following theorems of \cite{Ollivier}:

\begin{theorem}[\cite{Ollivier}, Thm.~2]\label{theorem:Ollivier_1}
Let $G$ be a random group at density $D$. Then for any positive $\epsilon$, and any
efficient van Kampen diagram $\D$ containing $m$ disks, we have
$$|\partial\D| \ge (1-2D-\epsilon)\cdot nm$$
with probability $1-O(e^{-n^c})$.
\end{theorem}

\begin{theorem}[\cite{Ollivier}, Cor.~3]\label{theorem:Ollivier_2}
Let $G$ be a random group at density $D$. Then the hyperbolicity constant $\delta$ of the
presentation satisfies
$$\delta \le 4n/(1-2D)$$
with probability $1-O(e^{-n^c})$.
\end{theorem}

From this, we will deduce the following lemma, which is the exact analog of
\cite{Calegari_Walker}  Lem.~6.3.2, and is deduced in exactly the same way from Ollivier's
theorems:

\begin{lemma}\label{lemma:bounded_faces}
Let $\overline{M}(Z)$ be a 3-manifold obtained from a beaded spine, and suppose it is not
$\pi_1$-injective in $G$, a random group at density $D$. Then there are constant $C$ and $C'$
depending only on $D<1/2$, a geodesic path $\gamma$ in $Z$ of length at most $Cn$, and
a van Kampen diagram $\D$ containing at most $C'$ faces so that $\gamma \subset \partial \D$
and $|\gamma|>|\partial \D|/2$.
\end{lemma}
\begin{proof}
Theorem~\ref{theorem:Ollivier_2} says that $\delta \le 4n/(1-2D)$, and in any
$\delta$-hyperbolic geodesic metric space, a $k$-local geodesic is a (global) 
$(\frac {k+4\delta}{k-4\delta},2\delta)$-quasigeodesic for any $k>8\delta$ (see
\cite{Bridson_Haefliger}, Ch.~III.~H, 1.13 p.~405). A local geodesic in the 1-skeleton of $\overline{M}(Z)$
corresponding to an element of the kernel must contain a subsegment of length at most $9\delta$ which is not a local geodesic
in $K$.  The lift of this subsegment to the universal cover $\widetilde{K}$ (i.e. the Cayley complex of $G$) cobounds a van Kampen diagram $\D$ with an honest geodesic segment in $G$. But
by Theorem~\ref{theorem:Ollivier_1}, the diagram $\D$ must satisfy
$$72n/(1-2D)\ge |\partial D|\ge (1-2D-\epsilon)\cdot nC'$$
where $C'$ is the number of faces; thus $C'$ is bounded in terms of $D$, and independent of
$n$.
\end{proof}

\subsection{Proof of the main theorem}

\begin{theorem}[3-Manifolds Everywhere]\label{theorem:random_acylindrical_subgroup}
Fix $k\ge 2$. A random $k$-generator group --- either in the few relators model
with $\ell \ge 1$ relators, or the density model with density $0<D<1/2$ --- with relators
of length $n$ contains many quasi-isometrically embedded subgroups isomorphic to the fundamental group 
of a hyperbolic 3-manifold with totally geodesic boundary, with probability $1-O(e^{-n^C})$ for some $C>0$.
\end{theorem}
\begin{proof}
The proof exactly follows the proof of Thm.~6.4.1 from \cite{Calegari_Walker}.
Pick one relation $r$ and build $L \to Z \to X$ and $\overline{M}(Z)$ by the method of
\S\ \ref{section:thin_spine_section}. We have already seen that $\overline{M}(Z)$ is homotopy equivalent to 
a hyperbolic 3-manifold with totally geodesic boundary, and that its fundamental group
injects into $\langle F\; | \; r\rangle$; we now show that it stays injective in
$G$ when we add another $(2k-1)^{Dn}$ independent random relations of length $n$.

The argument is a straightforward application of Gromov's intersection formula,
i.e.\/ Proposition~\ref{proposition:intersection}. It is convenient to express it in
terms of {\em degrees of freedom}, measured multiplicatively as powers of $(2k-1)$.
Suppose $\overline{M}(Z)$ is not $\pi_1$-injective. Then by Lemma~\ref{lemma:bounded_faces}
there is an efficient 
van Kampen diagram $\D$ with $m \le C'$ faces (where $C'$ depends only on $D$), 
and a local geodesic $\gamma$ which immerses in $Z$ with $\gamma \subset \partial \D$ and 
$|\gamma|>|\partial \D|/2$. The choice of $\gamma$ gives $n\beta'$ degrees of freedom,
where $\beta' = \log(3)\alpha/\lambda$ and where $|\gamma|=\alpha n$, since there are
$|Z|\cdot 3^{\alpha n/\lambda}$ immersed paths in $Z$ of length $\alpha n$, and we have
$|Z|=O(n)$. Taking $\lambda$ as big as necessary, we can make $\beta'$ as small as we like.

Next, disks in $\D$ with boundary label $r$ or $r^{-1}$ cannot have segments of length
more than $\beta n$ in common either with themselves or with $\gamma$ in an efficient
van Kampen diagram, where $\beta$ as in Lemma~\ref{lemma:common_path_lifts} 
can be taken as small as we like. Since there are at most $m$ disks in $\D$, two distinct
disks cannot have more than $m$ boundary segments in common. Take $\beta$ small enough
so that $\beta m < 1/2$. Then if we let $\D'$ denote
the result of cutting the disks labeled $r$ or $r^{-1}$ out of the diagram, and let
$\gamma'\subset \partial \D'$ denote the union 
$$\gamma':=(\gamma \cap \partial \D') \cup (\partial \D' - \partial \D)$$
and $m'$ the number of disks in $\D'$, then we have inequalities $m' \le m$,
$|\gamma'|\ge |\gamma|$ and $|\gamma'|\ge |\partial \D'|/2$ with equality if and only if
$\D' = \D$.

Each remaining choice of face gives $nD$ degrees of freedom, and each segment in the
interior of length $\ell$ imposes $\ell$ degrees of constraint. Similarly,
$\gamma'$ itself imposes $|\gamma'|$ degrees of constraint. Let $\I$ denote the union of
interior edges. Then $|\partial \D'| + 2|\I|=nm'$ so $|\gamma'|+|\I|\ge nm'/2$ because
$|\gamma'|\ge |\partial \D'|/2$. On the other hand, the total degrees of freedom is
$nm'D + n\beta' < nm'/2$ if $\beta'$ is small enough, so there is no way to assign labels
to the faces to build a compatible diagram, with probability $1-O(e^{-n^C})$. There
are polynomial in $n$ ways to assign lengths to the edges, and a finite number of possible
combinatorial diagrams (since each diagram has at most $C'$ disks); summing the exceptional
cases over all such diagrams shows that the probability of finding some such diagram is
$O(e^{-n^C})$. Otherwise $\overline{M}(Z)$ is $\pi_1$-injective, as claimed.

Finally, we prove that $\overline{M}(Z)$ is quasi-isometrically embedded in $G$.  Indeed, as observed already in \cite{Calegari_Walker}, the above argument actually shows that for any $\epsilon>0$ we can construct $Z$ for which $\pi_1(\overline{M}(Z))$ is $(1+\epsilon)$ quasi-isometrically embedded in $G$.   Controlling $\epsilon$ depends only on applying the argument above to segments $\gamma$ of length at most $\alpha n$ for suitable $\alpha(\epsilon)$.  The constant $\alpha$ then bounds the number of disks in an efficient van Kampen diagram, by Theorem~\ref{theorem:Ollivier_1}.
\end{proof}

\section{Commensurability}\label{section:commensurability}

Once we know that random groups contain many interesting 3-manifold subgroups, it is natural
to wonder exactly {\em which} 3-manifold groups arise. For any fixed non-free
finitely presented group $H$, there are no injective homomorphisms from $H$ to a random group
$G$ at fixed density once the relators in $G$ get sufficiently long (with overwhelming
probability). So it is probably hopeless to try to understand {\em precisely} which subgroups
arise, since this will depend in a very complicated way on the length $n$ of the relators.
However, we are in better shape if we simply try to control the {\em commensurability class}
of the subgroups.

\begin{example}
There are 24 simplices in the barycentric subdivision of a regular Euclidean tetrahedron. Three edges
of this simplex have dihedral angles $\pi/2$, two have dihedral angles $\pi/3$, and one (the edge
lying on the edge of the original simplex) has a dihedral angle of the form $\frac 1 2 \cos^{-1}(\frac 1 3)$,
which is approximately $35.2644^\circ$. If we deform the dihedral angle $\alpha$
of this last edge while keeping the other dihedral angles fixed, the simplex admits a unique 
hyperbolic metric for all $\alpha > \pi/6$ at which point one vertex of the simplex
becomes ideal. If we try to deform to $\alpha < \pi/6$, then three of the faces of the
simplex don't meet at all, and there is a perpendicular plane which intersects these three faces
in the edges of a hyperbolic triangle with angles $(\pi/2,\pi/3,\alpha)$. 
If $\alpha$ is of the form $\pi/m$ for some integer $m>6$,
the group generated by reflections in the 4 sides of the (now-infinite) hyperbolic polyhedron
is discrete, and convex cocompact. This group is a Coxeter group which we denote
$\Gamma(m)$, and whose Coxeter diagram is
\begin{center}
\begin{picture}(70,10)(0,0)
\put(5,5){\circle*{3}}
\put(35,5){\circle*{3}}
\put(65,5){\circle*{3}}
\put(95,5){\circle*{3}}
\put(5,5){\line(1,0){90}}
\put(16,9){$m$}
\end{picture}
\end{center}
The groups $\Gamma(m)$ with $m<6$ are finite.
The group $\Gamma(6)$ is commensurable with the fundamental group of the figure 8 knot complement.
For any $m\ge 7$ the limit set of $\Gamma(m)$ is a (round) Sierpinski carpet. 
Figure~\ref{7fold_carpet} depicts a simply-connected 2-complex $\tilde{K}$ stabilized
by $\Gamma(7)$ with cocompact fundamental domain which is a ``dual'' spine to the
(infinite) polyhedron described above. The faces of this 2-complex are regular 7-gons, and
the vertices are all 4-valent with tetrahedral symmetry.
\begin{figure}[htpb]
\labellist
\small\hair 2pt
%\pinlabel $\text{Type A}$ at -200 0
\endlabellist
\centering
\includegraphics[scale=0.1]{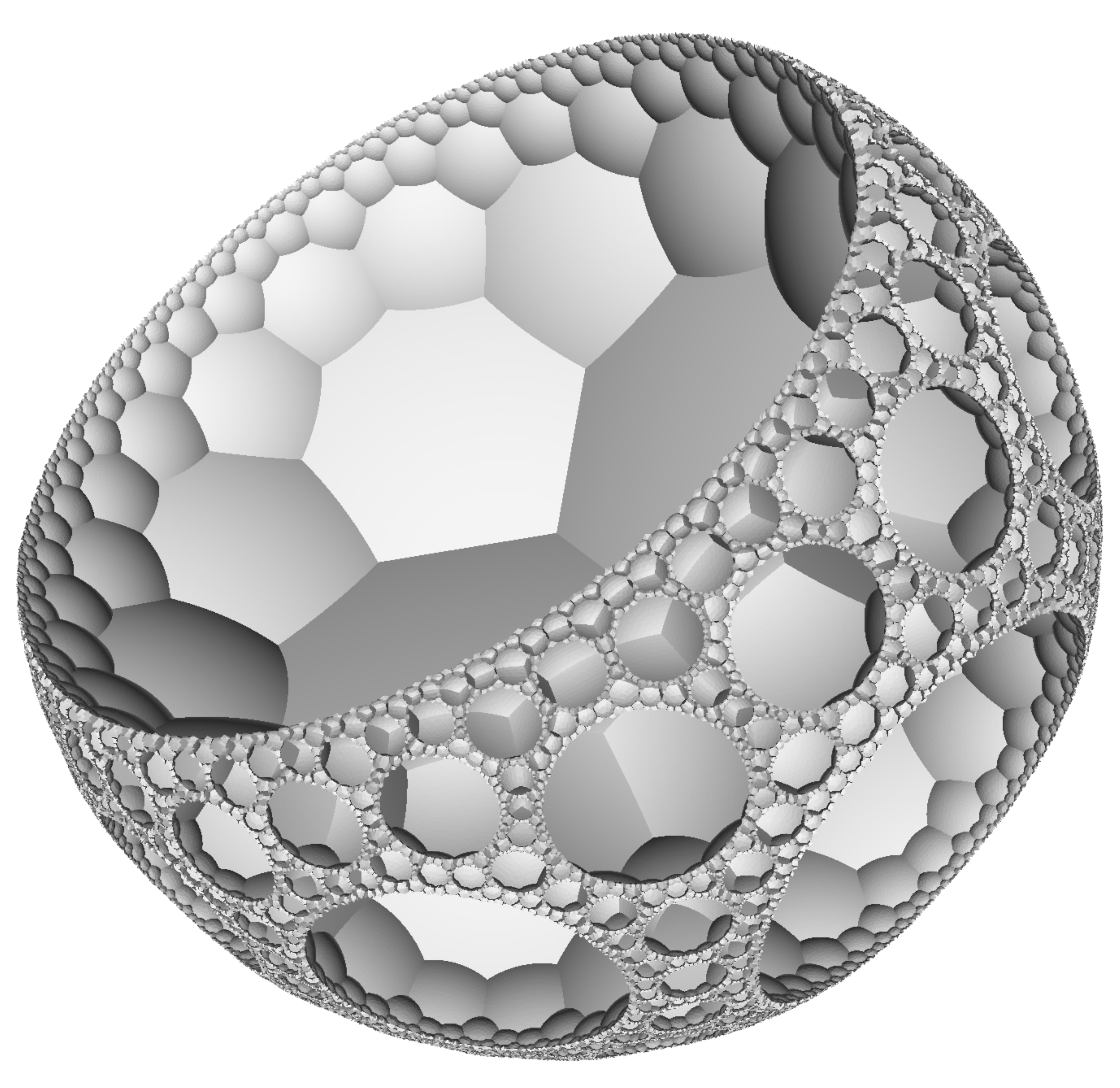}
\caption{A polyhedron $\tilde{K}$ on which $\Gamma(7)$ acts cocompactly}\label{7fold_carpet}
\end{figure}
As $m \to \infty$,
the limit sets converge to the Apollonian gasket, which has Hausdorff dimension about $1.3057$.
The convex covolumes of $\Gamma(m)$ --- i.e.\/ the volumes of the convex hulls of $\H^3/\Gamma(m)$
--- are uniformly bounded above independently of $m$.
\end{example}

We now observe that we can arrange for the 3-manifold groups we construct to be
commensurable with some $\Gamma(m)$.

\begin{theorem}[Commensurability Theorem]\label{theorem:commensurability_theorem}
A random group at any density $<1/2$ or in the few relators model contains 
(with overwhelming probability) a subgroup commensurable with the Coxeter group $\Gamma(m)$
for some $m\ge 7$, where $\Gamma(m)$ is the Coxeter group with Coxeter diagram
\begin{center}
\begin{picture}(70,10)(0,0)
\put(5,5){\circle*{3}}
\put(35,5){\circle*{3}}
\put(65,5){\circle*{3}}
\put(95,5){\circle*{3}}
\put(5,5){\line(1,0){90}}
\put(16,9){$m$}
\end{picture}
\end{center}
\end{theorem}
\begin{proof}
The group $\Gamma(m)$ acts cocompactly on a simply-connected 
2-dimensional complex $\tilde{K}$ in $\H^3$ whose faces are totally geodesic regular 
hyperbolic $m$-gons, and whose vertices are 4-valent and are stabilized by a tetrahedral
symmetry group; the case $m=7$ is depicted in Figure~\ref{7fold_carpet}. 
So to prove the theorem it suffices to show that we can build our thin
spines $L \to Z$ in such a way that each component of $L$ maps over exactly $m$ edges of
the 4-valent graph $Z$. But this is elementary to arrange: the only point in the construction in
which the number of edges of the components of $L$ might vary is during the operations
of super-compatible gluing, the adjustments in \S~\ref{subsection:distribution}, and trades.
In each of these cases all that is relevant is the types of pieces being glued or traded, and
not which components of $L$ are involved. Since types of the desired kind for each move can
be found on any component, we can simply distribute the moves evenly over the different 
components, possibly after taking multiple copies of $L$ to clear denominators. The proof
immediately follows
\end{proof}

There is nothing very special about the commensurability classes $\Gamma(m)$, except that their
fundamental domains are so small, so that their local combinatorics are very easy to describe.

\begin{definition}
A {\em geodesic spine} $K$ is a finite 2-dimensional complex with totally geodesic
edges and faces which embeds in some hyperbolic 3-manifold $M$ with totally geodesic boundary
as a deformation retract. We say that a 2-dimensional orbifold complex $K'$ is obtained by
{\em orbifolding} $K$ if its underlying complex is homeomorphic to $K$, and it is obtained by
adding at most one orbifold point to each face of $K$. Note that each such $K'$ has an
(orbifold) fundamental group which is commensurable with an acylindrical 3-manifold group.
\end{definition}

In view of the level of control we are able to impose on the combinatorial type of the thin
spines we construct in \S~\ref{section:thin_spine_section}, we make the following conjecture:

\begin{conjecture}
For any fixed geodesic spine $K$, a random group $G$ --- either in the few relators model or the
density model with density $0<D<1/2$ --- contains subgroups commensurable with the
(orbifold) fundamental group of some orbifolding $K'$ of $K$ (with overwhelming probability).
\end{conjecture}

\section{Acknowledgments}
Danny Calegari was supported by NSF grant DMS 1005246. Henry Wilton was supported by an
EPSRC Career Acceleration Fellowship. We would like to thank Fran\c cois Dahmani for
raising the main question addressed by this paper, and for his continued interest.
We would also like to thank Pierre de la Harpe for comments and corrections.

\end{document}

%% file: fourbizenesannotated.pdf_tex
%% Creator: Inkscape inkscape 0.48.2, www.inkscape.org
%% Accompanies image file 'fourbizenesannotated.pdf' (pdf, eps, ps)
\begingroup%
  \makeatletter%
  \providecommand\color[2][]{%
    \errmessage{(Inkscape) Color is used for the text in Inkscape, but the package 'color.sty' is not loaded}%
    \renewcommand\color[2][]{}%
  }%
  \providecommand\transparent[1]{%
    \errmessage{(Inkscape) Transparency is used (non-zero) for the text in Inkscape, but the package 'transparent.sty' is not loaded}%
    \renewcommand\transparent[1]{}%
  }%
  \providecommand\rotatebox[2]{#2}%
  \ifx\svgwidth\undefined%
    \setlength{\unitlength}{851.76293945bp}%
    \ifx\svgscale\undefined%
      \relax%
    \else%
      \setlength{\unitlength}{\unitlength * \real{\svgscale}}%
    \fi%
  \else%
    \setlength{\unitlength}{\svgwidth}%
  \fi%
  \global\let\svgwidth\undefined%
  \global\let\svgscale\undefined%
  \makeatother%
  \begin{picture}(1,0.92766422)%
    \put(0,0){\includegraphics[width=\unitlength]{fourbizenesannotated.pdf}}%
    \put(0.12751388,0.3719155){\color[rgb]{0,0,0}\makebox(0,0)[lb]{\smash{$e_\bullet$}}}%
    \put(0.7943053,0.54724711){\color[rgb]{0,0,0}\makebox(0,0)[lb]{\smash{$e_\bullet$}}}%
  \end{picture}%
\endgroup%

%% file: 3manifolds_everywhere.bbl
\begin{thebibliography}{00}
\bibitem[Agol(2013)]{Agol_VHC}
	I. Agol (with an appendix with D. Groves and J. Manning),
	\emph{The virtual Haken conjecture},
	Doc. Math. {\bf 18} (2013), 1045--1087
\bibitem[Bonahon(2002)]{Bonahon_Structures}
	F. Bonahon
	\emph{Geometric structures on 3-manifolds},
	Handbook of geometric topology, 
	North-Holland, Amsterdam, 2002, 93--164 
\bibitem[Bowditch(1998)]{Bowditch}
	B. Bowditch,
	\emph{Cut points and canonical splittings of hyperbolic groups},
	Acta Math. {\bf 180} (1998), no. 2, 145--186
\bibitem[Bridson--Haefliger(1999)]{Bridson_Haefliger}
	M. Bridson and A. Haefliger,
	\emph{Metric spaces of nonpositive curvature},
	Grund. der math. Wiss. Springer-Verlag Berlin 1999
\bibitem[Calegari--Walker(2013)]{Calegari_Walker_RR}
	D. Calegari and A. Walker,
	\emph{Random rigidity in the free group},
	Geom. Topol. {\bf 17} (2013), 1707--1744
\bibitem[Calegari--Walker(2015)]{Calegari_Walker}
	D. Calegari and A. Walker,
	\emph{Random groups contain surface subgroups},
	Jour. AMS {\bf 28} (2015), no. 2, 383--419
\bibitem[Dahmani--Guirardel--Przytycki(2011)]{Dahmani_Guirardel_Przytycki}
	F. Dahmani, V. Guirardel, and P. Przytycki.
	\emph{Random groups do not split},
	Math. Ann. {\bf 249} (2011), no. 3, 657--673
\bibitem[Gromov(1993)]{Gromov}
	M. Gromov,
	\emph{Asymptotic invariants of infinite groups},
	Geometric group theory, vol. 2; LMS Lecture Notes {\bf 182} (Niblo and Roller eds.)
	Cambridge University Press, Cambridge, 1993
\bibitem[Kahn--Markovic(2011)]{Kahn_Markovic}
	J. Kahn and V. Markovic,
	\emph{The good pants homology and a proof of the Ehrenpreis conjecture},
	preprint, arXiv:1101.1330
\bibitem[Kapovich--Kleiner(2000)]{Kapovich_Kleiner}
	M. Kapovich and B. Kleiner,
	\emph{Hyperbolic groups with low dimensional boundary},
	Ann. Sci. \'Ecole Norm. Sup. (4) {\bf 33} (2000), no.\ 5, 647--669
\bibitem[Keevash(2014)]{Keevash}
	P. Keevash,
	\emph{The existence of designs},
	preprint, arXiv:1401.3665
\bibitem[Lyndon--Schupp(1977)]{Lyndon_Schupp}
	R. Lyndon and P. Schupp,
	\emph{Combinatorial group theory},
	Ergeb. der Math. und ihr. Grenz. {\bf 89}, Springer-Verlag, Berlin-New York, 1977
\bibitem[Ollivier(2007)]{Ollivier}
	Y. Ollivier,
	\emph{Some small cancellation properties of random groups},
	Internat. J. Algebra Comput. {\bf 17} (2007), no. 1, 37--51
\bibitem[Stallings(1968)]{Stallings}
	J. Stallings,
	\emph{On torsion-free groups with infinitely many ends},
	Ann. Math. {\bf 88} (1968), 312--334
\bibitem[Stallings(1983)]{Stallings_graphs}
	J. Stallings,
	\emph{Topology of finite graphs},
	Invent.\ Math.\ {\bf 71} (1983), no.\ 3, 551--565
\end{thebibliography}
